\documentclass[english]{article}
\usepackage[T1]{fontenc}
\usepackage[english]{babel}
\usepackage[latin9]{inputenc}
\usepackage{geometry}
\usepackage{amsmath}
\usepackage{amsthm}
\usepackage{amssymb}
\usepackage{esint}
\usepackage{hyperref}
\makeatletter
\theoremstyle{plain}
\newtheorem{thm}{\protect\theoremname}
  \theoremstyle{definition}
  \newtheorem{defn}[thm]{\protect\definitionname}
  \theoremstyle{plain}
  \newtheorem{prop}[thm]{\protect\propositionname}
  \theoremstyle{plain}
  \newtheorem{cor}[thm]{\protect\corollaryname}
  \theoremstyle{plain}
  \newtheorem{lem}[thm]{\protect\lemmaname}
  \theoremstyle{remark}
  \newtheorem*{remark}{Remark}

\makeatother

\usepackage{babel}
  \providecommand{\definitionname}{Definition}
  \providecommand{\lemmaname}{Lemma}  \providecommand{\propositionname}{Proposition}
  \providecommand{\corollaryname}{Corollary}
\providecommand{\theoremname}{Theorem}

\newcommand{\R}{{\mathbb R}}
\newcommand{\Z}{{\mathbb Z}}
\newcommand{\E}[1]{{\mathbb E}\left [#1\right]}

\newcommand{\ep}{\varepsilon}

\newcommand{\x}{\ensuremath{\times}}

\newcommand{\abs}[1]{\ensuremath{\left| #1 \right|}}
\newcommand{\mat}{}
\newcommand{\lr}[1]{\ensuremath{\left(#1 \right)}}
\newcommand{\norm}[1]{\left\lVert#1\right\rVert}
\newcommand{\inprod}[2]{\ensuremath{\left\langle#1,#2\right\rangle}}

\newcommand{\set}[1]{\ensuremath{\{#1\}}}

\def\XXint#1#2#3{{\setbox0=\hbox{$#1{#2#3}{\int}$} \vcenter{\hbox{$#2#3$}}\kern-.5\wd0}} 

\newcommand{\Jac}{\mat{\mathop{\mathrm{Jac}}}}
\DeclareMathOperator{\act}{\mathrm{act}}
\DeclareMathOperator{\Act}{\mathrm{Act}}

\DeclareMathOperator{\Relu}{ReLU}

\title{Products of Many Large Random Matrices and Gradients in Deep Neural Networks}

\author{Boris Hanin\footnote{Department of Mathematics, Texas A\&M; bhanin@tamu.edu} \and Mihai Nica\footnote{Department of Mathematics, University of Toronto; mnica@math.utoronto.edu}}

\begin{document}

\setcounter{section}{0}

\maketitle
\global\long\def\p{\mathbf{P}}
\global\long\def\q{\mathbf{Q}}
\global\long\def\cov{\mathbf{Cov}}
\global\long\def\var{\mathbf{Var}}
\global\long\def\corr{\mathbf{Corr}}
\global\long\def\e{\mathbf{E}}
\global\long\def\one{\mathtt{1}}

\global\long\def\pp#1{\mathbf{P}\left(#1\right)}
\global\long\def\ee#1{\mathbf{E}\left[#1\right]}
\global\long\def\norm#1{\left\Vert #1\right\Vert }
\global\long\def\abs#1{\left|#1\right|}
\global\long\def\given#1{\left|#1\right.}
\global\long\def\ceil#1{\left\lceil #1\right\rceil }
\global\long\def\floor#1{\left\lfloor #1\right\rfloor }

\global\long\def\bA{\mathbb{A}}
\global\long\def\bB{\mathbb{B}}
\global\long\def\bC{\mathbb{C}}
\global\long\def\bD{\mathbb{D}}
\global\long\def\bE{\mathbb{E}}
\global\long\def\bF{\mathbb{F}}
\global\long\def\bG{\mathbb{G}}
\global\long\def\bH{\mathbb{H}}
\global\long\def\bI{\mathbb{I}}
\global\long\def\bJ{\mathbb{J}}
\global\long\def\bK{\mathbb{K}}
\global\long\def\bL{\mathbb{L}}
\global\long\def\bM{\mathbb{M}}
\global\long\def\bN{\mathbb{N}}
\global\long\def\bO{\mathbb{O}}
\global\long\def\bP{\mathbb{P}}
\global\long\def\bQ{\mathbb{Q}}
\global\long\def\bR{\mathbb{R}}
\global\long\def\bS{\mathbb{S}}
\global\long\def\bT{\mathbb{T}}
\global\long\def\bU{\mathbb{U}}
\global\long\def\bV{\mathbb{V}}
\global\long\def\bW{\mathbb{W}}
\global\long\def\bX{\mathbb{X}}
\global\long\def\bY{\mathbb{Y}}
\global\long\def\bZ{\mathbb{Z}}

\global\long\def\cA{\mathcal{A}}
\global\long\def\cB{\mathcal{B}}
\global\long\def\cC{\mathcal{C}}
\global\long\def\cD{\mathcal{D}}
\global\long\def\cE{\mathcal{E}}
\global\long\def\cF{\mathcal{F}}
\global\long\def\cG{\mathcal{G}}
\global\long\def\cH{\mathcal{H}}
\global\long\def\cI{\mathcal{I}}
\global\long\def\cJ{\mathcal{J}}
\global\long\def\cK{\mathcal{K}}
\global\long\def\cL{\mathcal{L}}
\global\long\def\cM{\mathcal{M}}
\global\long\def\cN{\mathcal{N}}
\global\long\def\cO{\mathcal{O}}
\global\long\def\cP{\mathcal{P}}
\global\long\def\cQ{\mathcal{Q}}
\global\long\def\cR{\mathcal{R}}
\global\long\def\cS{\mathcal{S}}
\global\long\def\cT{\mathcal{T}}
\global\long\def\cU{\mathcal{U}}
\global\long\def\cV{\mathcal{V}}
\global\long\def\cW{\mathcal{W}}
\global\long\def\cX{\mathcal{X}}
\global\long\def\cY{\mathcal{Y}}
\global\long\def\cZ{\mathcal{Z}}

\global\long\def\sA{\mathscr{A}}
\global\long\def\sB{\mathscr{B}}
\global\long\def\sC{\mathscr{C}}
\global\long\def\sD{\mathscr{D}}
\global\long\def\sE{\mathscr{E}}
\global\long\def\sFA{\mathscr{F}}
\global\long\def\sG{\mathscr{G}}
\global\long\def\sH{\mathscr{H}}
\global\long\def\sI{\mathscr{I}}
\global\long\def\sJ{\mathscr{J}}
\global\long\def\sK{\mathscr{K}}
\global\long\def\sL{\mathscr{L}}
\global\long\def\sM{\mathscr{M}}
\global\long\def\sN{\mathscr{N}}
\global\long\def\sO{\mathscr{O}}
\global\long\def\sP{\mathscr{P}}
\global\long\def\sQ{\mathscr{Q}}
\global\long\def\sR{\mathscr{R}}
\global\long\def\sS{\mathscr{S}}
\global\long\def\sT{\mathscr{T}}
\global\long\def\sU{\mathscr{U}}
\global\long\def\sV{\mathscr{V}}
\global\long\def\sW{\mathscr{W}}
\global\long\def\sX{\mathscr{X}}
\global\long\def\sY{\mathscr{Y}}
\global\long\def\sZ{\mathscr{Z}}

\global\long\def\tr{\text{Tr}}
\global\long\def\re{\text{Re}}
\global\long\def\im{\text{Im}}
\global\long\def\supp{\text{supp}}
\global\long\def\sgn{\text{sgn}}
\global\long\def\d{\text{d}}
\global\long\def\dist{\text{dist}}
\global\long\def\spn{\text{span}}
\global\long\def\ran{\text{ran}}
\global\long\def\ball{\text{ball}}
\global\long\def\ai{\text{Ai}}
\global\long\def\occ{\text{Occ}}

\global\long\def\To{\Rightarrow}
\global\long\def\half{\frac{1}{2}}
\global\long\def\oo#1{\frac{1}{#1}}

\global\long\def\al{\alpha}
\global\long\def\be{\beta}
\global\long\def\ga{\gamma}
\global\long\def\Ga{\Gamma}
\global\long\def\de{\delta}
\global\long\def\De{\Delta}
\global\long\def\ep{\epsilon}
\global\long\def\ze{\zeta}
\global\long\def\et{\eta}
\global\long\def\th{\theta}
\global\long\def\Th{\Theta}
\global\long\def\ka{\kappa}
\global\long\def\la{\lambda}
\global\long\def\La{\Lambda}
\global\long\def\rh{\rho}
\global\long\def\si{\sigma}
\global\long\def\ta{\tau}
\global\long\def\ph{\phi}
\global\long\def\Ph{\Phi}
\global\long\def\vp{\varphi}
\global\long\def\ch{\chi}
\global\long\def\ps{\psi}
\global\long\def\Ps{\Psi}
\global\long\def\om{\omega}
\global\long\def\Om{\Omega}
\global\long\def\Si{\Sigma}

\global\long\def\dequal{\stackrel{d}{=}}
\global\long\def\defequal{\stackrel{\De}{=}}
\global\long\def\pto{\stackrel{\p}{\to}}
\global\long\def\asto{\stackrel{\text{a.s.}}{\to}}
\global\long\def\dto{\stackrel{\text{d.}}{\to}}
\global\long\def\ld{\ldots}
\global\long\def\di{\partial}
\global\long\def\pr{\prime}
\global\long\def\a{\text{act}}
\global\long\def\A{\text{Act}}
\global\long\def\X{W}
\global\long\def\B{D}
\global\long\def\Z{J}
\global\long\def\S{S}
\global\long\def\F{F}
\global\long\def\ra{\rightarrow}
\begin{abstract}
We study products of random matrices in the regime where the number of terms and the size of the matrices simultaneously tend to infinity. Our main theorem is that the logarithm of the $\ell_2$ norm of such a product applied to any fixed vector is asymptotically Gaussian. The fluctuations we find can be thought of as a finite temperature correction to the limit in which first the size and then the number of matrices tend to infinity. Depending on the scaling limit considered, the mean and variance of the limiting Gaussian depend only on either the first two or the first four moments of the measure from which matrix entries are drawn. We also obtain explicit error bounds on the moments of the norm and the Kolmogorov-Smirnov distance to a Gaussian. Finally, we apply our result to obtain precise information about the stability of gradients in randomly initialized deep neural networks with ReLU activations. This provides a quantitative measure of the extent to which the exploding and vanishing gradient problem occurs in a fully connected neural network with ReLU activations and a given architecture. \end{abstract}

\section{Introduction}
Products of independent random matrices are a classical topic in probability and mathematical physics with applications to a variety of fields, from wireless communication networks \cite{tulino2004random} to the physics of black holes \cite{cotler2017black}, random dynamical systems \cite{pollicott2010maximal}, and recently to the numerical stability of randomly initialized neural networks \cite{hanin2018neural,pennington2017resurrecting}. In the context of neural networks, products of random matrices are related to the numerical stability of gradients at initialization and therefore give precise information about the exploding and vanishing gradient problem (see Section \ref{S:neural-nets}, Proposition \ref{P:even-grads}, and Corollary \ref{C:EVGP}). The purpose of this article is to prove several new results about such products in the regime where the number of terms and the sizes of matrices grow simultaneously. This regime has attracted attention \cite{dsl_2,joint_limit} but remains poorly understood. We find new phenomena not present when the number of terms and the size of the matrices are sent to infinity sequentially rather than simultaneously (see Section \ref{S:scaling-limits} for more on this point). 

To explain our results, let $d \in \bN$ be a positive integer and let $n_0,\ld,n_d\in \bN$ be a list of positive integers. We are concerned with (the non-asymptotic) analysis of products of $d$ independent rectangular random matrices of sizes $n_i \times n_{i-1}$ with real entries:
\begin{equation}
\mat{M}^{(d)}~=~\mat{M}^{(d)}(n_0,\ld,n_d)~\defequal~ \mat{X}^{(d)} \cdots \mat{X}^{(1)},\qquad \mat{X}^{(i)}\in \mathrm{Mat}(n_i,n_{i-1}). \label{E:J-def}
\end{equation}
The specific matrix ensembles we study depend on a parameter $p\in (0,1]$ and a distribution $\mu$ on $\bR$. We define: 
\begin{equation}\label{E:model-def}
\mat{X}^{(i)} \defequal \lr{ p n_{i-1} } ^{-\half} \mat{D}^{(i)}\mat{W}^{(i)}
\end{equation}
%,\qquad \mat{M}^{(d)}=\mat{X}^{(d)}\cdots \mat{X}^{(1)},
where $\mat{D}^{(i)}$ are $n_i \times n_i$ diagonal matrices 
\[\mat{D}^{(i)}=\mathrm{Diag}\lr{\xi_j^{(i)},\,\, j=1,\ldots, n_i} \in \mathrm{Mat}(n_i,n_i),\qquad \xi_j^{(j)}~\sim~\mathrm{Bernoulli}(p)\,\,i.i.d.,\]
and $\mat{W}^{(i)} \in \mathrm{Mat}(n_i,n_{i-1})$ are independent $n_i \times n_{i-1}$ random matrices for which the entries $W_{a,b}^{(i)}$ are drawn i.i.d. from a fixed distribution $\mu$ on $\R$ satisfying the following four conditions:
\begin{align}
\notag &(i)~ \text{normalization: } \E{W^{(i)}_{a,b}} = 0,\quad \E{\lr{W^{(i)}_{a,b}}^2} =1\,\quad\qquad  (ii)~ \text{symmetry around }0\text{:} \quad W^{(i)}_{a,b} \dequal -W^{(i)}_{a,b}\\
& (iii)~ \text{finite moments:}\quad \forall k\in \bN, \quad \E{ \lr{W^{(i)}_{a,b}}^k } \defequal \mu_k < \infty \quad (iv)~ \text{no atoms:}\quad \p\lr{W^{(i)}_{a,b} = x} = 0 \quad \forall x \in \bR.\label{E:mu-def}
\end{align} 
When $p=1,$ the matrices $\mat{D}^{(i)}$ are the identity. In contrast, when $p = 1/2$, the matrix product $M^{(d)}$ naturally arises in connection to the input-output Jacobian matrix  $\Jac^{(d)}$ for neural nets with $\Relu$ nonlinearity and $d$ layers with widths $n_0,\ldots,n_d$ initialized with random weights drawn from $\mu$. In particular, the following equality in distribution holds when $p =\half$:
  \[ \big(\Jac^{(d)}\big)^T\Jac^{(d)} ~\dequal~ \big(\mat{M}^{(d)}\big)^T\mat{M}^{(d)},\]
so that, when $p=1/2,$ the singular values of $\mat{M}^{(d)}$ are equal in distribution to those of $\Jac^{(d)}.$ This is a consequence of Proposition \ref{P:even-grads} below, which opens the door to a rigorous study of the so-called exploding and vanishing gradient problem for $\Relu$ nets at finite depth and width. This refines the approach of the first author in \cite{hanin2018neural}, and we refer the reader to Section \ref{S:neural-nets} for precise definitions an extended discussion of this point. Our main result concerns the distribution of 
\[Z_d(\vec{u})~\defequal ~\frac{n_0}{n_d}||\mat{M}^{(d)}\vec{u}||^2,\qquad \vec{u}\in \R^{n_0},\,\, \norm{\vec{u}}=1.\] 
As explained in Section \ref{S:RP} below, $Z_d(\vec{u})$ can be thought of as a  line-to-line partition function in a disordered medium given by the computation graph underlying the matrix product defining $\mat{M}^{(d)}$, with $\vec{u}$ corresponding to a kind of initial condition. The diagonal matrices $\mat{D}^{(i)}$ then correspond to $\{0,1\}-$valued spins on the vertices of this graph, restricting the allowed paths of the directed random polymer. With this interpretation, our main result, Theorem \ref{T:main}, shows that the analogue of the free energy, namely
\begin{equation}
\mathcal \ln\lr{ Z_{d}(\vec{u})}~=~ \ln\lr{\frac{n_0}{n_d}||\mat{M}^{(d)}\vec{u}||^2}\label{E:F-def}
\end{equation}
is Gaussian up to an error that tends to zero when $n_i,d$ tend to infinity. 

%There are several motivations behind the study of the ensembles \eqref{E:J-def}. First, the study of products of random matrices is a classical question of general interest in random matrix theory. We discuss this point of view in Section \ref{S:RMT} as well as below the statement of our main result, Theorem \ref{T:main}. The second motivation is due to the connection between this ensemble and deep neural networks. We prove in Section \ref{S:neural-nets} that singular values of $\mat{M}^{(d)}$ are equal in distribution to the singular values of the input-output Jacobian for a randomly initialized deep neural networks with depth $d$, layer widths $n_0,\ld, n_d$, and $\Relu$ activations when $p = 1/2$. The numerical stability of the start of gradient-based training for such a network is therefore related to the fluctuations of random variables such as $ \big|\big|\mat{M}^{(d)}\vec{u}\big|\big|_2^2$, which are characterized by Theorem \ref{T:main}. 

 \begin{thm} \label{T:main}
Fix $p\in (0,1]$, and a distribution $\mu$ satisfying (i)-(iv) above. Let $\vec{u} \in \bR^{n_0}$ be some fixed unit vector, $\norm{\vec{u}}_2 = 1,$ and for any choice of $d$, $n_0,\ldots,n_d$, set
%\[\beta=\lr{\frac{3}{p}-1} \sum_{j=1}^{d}\frac{1}{n_j} + \frac{\mu_4 - 3}{p n_0} \norm{\vec{u}}^4_4.\]
\[\beta ~\defequal~ \left(\frac{3}{p}-1\right) \sum_{i=1}^{d} \frac{1}{n_i} + \frac{\mu_4-3}{p n_1}\norm{\vec{u}}_4^4 .\]
Let $\mat{M}^{(d)} = \mat{X}^{(d)} \cdots \mat{X}^{(1)}$ be as in \eqref{E:model-def}. Then, the norm of the vector $\mat{M}^{(d)}\vec{u}$ is approximately log-normal distributed: 
%\[\frac{n_0}{n_d}\norm{\mat{M}^{(d)}\vec{u}}_2^2~\approx~ \exp\lr{\sqrt{\beta}\cN(0,1)-\half \beta}.\]
\[\frac{n_0}{n_d}\norm{\mat{M}^{(d)}\vec{u}}_2^2~\approx~ \exp\lr{\cN\lr{-\half \beta,\, \beta}}.\]
This approximation holds both in the sense of distribution and of moments. More precisely, with $d_{\mathrm{KS}}$ denoting Kolomogov-Smirnov distance, 
%  \begin{equation}\label{E:KS-dist}
%d_{\mathrm{KS}}\lr{\ln\lr{\frac{n_0}{n_d} \norm{\mat{M}^{(d)}\vec{u}}_2^2}, \sqrt{\be}\cN(0,1) - \half \beta} = O\lr{\sum_{i=1}^{d} \frac{1}{n_i^2} }^{1/5}
%\end{equation}
  \begin{equation}\label{E:KS-dist-weak}
d_{\mathrm{KS}}\lr{\ln\lr{\frac{n_0}{n_d} \norm{\mat{M}^{(d)}\vec{u}}_2^2},~ \,\cN\lr{-\half \beta,\, \beta}} = O\lr{\sum_{i=1}^{d} \frac{1}{n_i^2} }^{1/5},
\end{equation}
where the implicit constant is uniform for $p$ in a compact subset of $(0,1]$, and $\beta$ in a compact set bounded away from $\beta = 0$. Moreover, for every $k\geq 0,$ satisfying $\binom{k}{2} < \min_{1 \leq i\leq d} n_i,$ we have
%\begin{equation}
%\e\left[\frac{n^k_0}{n_d^k}\norm{\mat{M}\vec{u}}^{2k}_2 \right] ~=~ \exp\left[\binom{k}{2}\beta+ O \left(\sum_{i=1}^{d} \frac{1}{n_i^2}\right)\right] ~=~ \E{ \lr{\exp\lr{ \sqrt{\be}\cN(0,1) - \half \beta}}^k + O\lr{\beta^{-1}\sum_{i=1}^{d} \frac{1}{n_i^2}}} ,\label{E:Mu-moments}
%\end{equation}
\begin{equation}
\e\left[\frac{n^k_0}{n_d^k}\norm{\mat{M}\vec{u}}^{2k}_2 \right] ~=~ \exp\left[\binom{k}{2}\beta+ O \left(\sum_{i=1}^{d} \frac{1}{n_i^2}\right)\right] ~=~ \E{ \lr{\exp\lr{\cN\lr{-\half \beta,\, \beta}}}^k} + O\lr{\beta^{-1}\sum_{i=1}^{d} \frac{1}{n_i^2}} ,\label{E:Mu-moments}
\end{equation}
where the implicit constant depends on $k$ and the moments of $\mu$ but not on $\beta, d,n_i.$
%     \begin{equation}
%\E{\frac{n_0^k}{n_d^k}\norm{J_p^{(d)}u}_2^{2k}}=\exp\lr{\binom{k}{2}\beta}\lr{1+ O\lr{\beta^{-1}\sum_{j=1}^{d-1}n_j^{-2}}}.\label{E:moments}
%\end{equation}

  \end{thm}
  
  \begin{remark}
  In the proof of equation \eqref{E:KS-dist-weak}, we actually show that $d_{\mathrm{KS}}\lr{\ln\lr{\frac{n_0}{n_d} \norm{\mat{M}^{(d)}\vec{u}}_2^2},\,\cN\lr{-\half \beta,\, \beta}}$ is bounded above by
    \begin{equation}\label{E:KS-dist} O\lr{\beta^{-1}\sum_{i=1}^{d-1}n_i^{-2}+\lr{\beta^{-2}\sum_{i=1}^{d-1}n_i^{-2}}^{1/5}+\lr{\beta^{-1/2}\sum_{i=1}^{d-1}n_i^{-2}}^{1/2}+\sum_{i=1}^d n_i^{-m}+\sum_{i=1}^dp^{n_i}}
\end{equation}
for any choice of $m \in \bN$ and where the constants depend only on $m,$ the moments of $\mu$ and $p$. By taking $m = 2$ and restricting the $\be$ in a compact set, the result claimed in Theorem \ref{T:main} holds. Moreover, if we take $p=1$, then we will actually prove instead the sharper result that 
  \begin{equation*}d_{\mathrm{KS}}\lr{\ln\lr{\frac{n_0}{n_d} \norm{\mat{M}^{(d)}\vec{u}}_2^2},\,\cN\lr{-\half \beta,\, \beta}}=O\lr{\beta^{-1}\sum_{i=1}^{d-1}n_i^{-2}+\lr{\beta^{-2}\sum_{i=1}^{d-1}n_i^{-2}}^{1/5}+\lr{\beta^{-1/2}\sum_{i=1}^{d-1}n_i^{-2}}^{1/2}}
\end{equation*}

  \end{remark}

\noindent The two conclusions, equation \eqref{E:KS-dist} and equation \eqref{E:Mu-moments}, of Theorem \ref{T:main} are proven separately and have independent proofs. We prove equation \eqref{E:Mu-moments} by a path-counting type argument in Section \ref{S:moments}. The argument in Section \ref{S:CLT} for equation \eqref{E:KS-dist}, in contrast, uses a central limit theorem for martingales.

\subsection{Joint scaling limits}\label{S:scaling-limits}
Theorem \ref{T:main} shows that the free energy $\ln(Z_d(\vec{u}))=\ln (||\mat{M}^{(d)}\vec{u}||_2^2)$ from \eqref{E:F-def} is Gaussian in the double scaling limit
\begin{equation}
d\to \infty,\qquad n_i=n_i(d)\to \infty,\qquad 0~<~\limsup_{d\to \infty}\sum_{j=1}^d \frac{1}{n_i(d)}~<~\infty,\label{E:scaling-limit}
\end{equation}
achieved for instance when $n_i(d)$ are equal and proportional to $d.$ This asymptotic normality for $\ln(Z_d(\vec{u}))$ \emph{cannot} be seen by taking the limits $d \to \infty$ and $\min\{n_i\}$ to infinity one after the other. Indeed, consider the case when $p=1$ and $\mu=\mathcal N(0,1),$ the standard Gaussian measure. A simple computation using the rotational invariance of i.i.d. Gaussian matrices shows the equality in distribution
\begin{equation}\label{E:chi_sq}
Z_d(\vec{u})~\dequal~\prod_{i=1}^d \chi_{n_i}^2/n_i,
\end{equation}
where $\chi_n^2$ is a chi-squared random variable with $n$ degrees of freedom and the terms in the product are independent. In the limit where $d$ is fixed and $\min\{n_i,\,i\geq 1\}\to \infty,$ we have $ \chi_{n_i}^2/n_i\approx 1 + O(n_i^{-1/2})$ and so 
\[\lim_{\min{n_i}\to \infty}\mathcal \ln (Z_{d}(\vec{u}))~=~0\qquad \text{almost surely}.\]
On the other hand, if the $n_i$ are uniformly bounded, then we have
\[\lim_{d\to \infty}\mathcal \ln (Z_{d}(\vec{u}))~=~\infty\qquad \text{almost surely}.\]
In fact, for $n_i\equiv n$ fixed, $\ln (Z_{d}(\vec{u}))$ converges only with an addition $1/d$ scaling:
\[\lim_{d\to \infty}\frac{1}{d}\,\ln (Z_{d}(\vec{u}))~=~\e\left[\log{||\mat{M}^{(1)}\vec{u}||^2}\right]\quad \text{almost surely}.\]
In particular, we have
\[\lim_{d\to \infty}\lim_{n_1,\ld,n_d\to \infty}\mathcal \ln (Z_{d}(\vec{u}))~\neq~ \lim_{n_1,\ld,n_d\to \infty}\lim_{d\to \infty}\ln (Z_{d}(\vec{u})),\]
making \eqref{E:scaling-limit} an interesting regime for $\ln(Z_{d}(\vec{u}))$. The non-commutativity of the $n_i,d\to\infty$ limits is well-known \cite{gaussian_RMT_6,DeiftOpenProblems,joint_limit} and is related to the fact that the local statistics of the singular values of $\mat{M}^{(d)}$ are sensitive to the order in which the limits above are taken. Remaining in the simple case of $p=1$ and $\mu$ Gaussian, a simple application of the central limit results show that when all the $n_i$ are equal and are related to $d$ by $n = 2\beta^{-1} d$, then the exact chi-squared representation of equation \eqref{E:chi_sq} gives the convergence in distribution:
\[\lim_{\substack{n_i = 2\be^{-1} d \\ d\to \infty}}\ln (Z_{d}(\vec{u}))~=~ \mathcal N(-\beta/2, \beta),\]
which is of course consistent with Theorem \ref{T:main}. Part of the content of Theorem \ref{T:main} is therefore that this result is essentially independent of the parameter $p$ and the measure $\mu$ according to which the entries of the matrices $\mat{W}^{(i)}$ are distributed. See Section \ref{S:contribution} for more discussion on the novel aspects of Theorem \ref{T:main}.% Thus, we interpret $\beta$ as the inverse temperature in the double scaling limit \eqref{E:double-scaling}, and Theorem \ref{T:main} shows that the implicit temperature of the model and the asymptotic normality are in fact universal with respect to $\mu$ and $p$. 
%We are interested specifically in quantitative estimates on the moments and distribution of the random variables 
%\[\frac{1}{n_d}\log \big|\big|\mat{M}^{(d)}\vec{u}\big|\big|_2^2,\qquad \vec{u}\in \R^{n_0},\,\,\norm{\vec{u}}=1\]
%that are valid for any given $d,n_0,\ld, n_d$ and capture the leading order behavior when $d,n_0,\ld, n_d$ simultaneously tend to infinity. 

\subsection{Connection to previous work in Random Matrix Theory} \label{S:RMT}
The literature on products of random matrices is vast. Much of the previous work concerns products of $d$ i.i.d. random matrices, each of size $n\x n$. Such ensembles have been well studied in two distinct regimes: (a) when $n$ is fixed and $d\to \infty$ and (b) when $d$ is fixed and $n\to \infty$. Case (a) is related to multiplicative ergodic theory and the study of Lyapunov exponents. The seminal articles in this regime are the results of Furstenberg and Kesten \cite{furstenberg1960}, which gives general conditions for the existence of the top Lyapunov exponent 
\[\la_{\max} = \lim_{d\to\infty} \frac{1}{d} \log\norm{\mat{M}^{(d)}}_{\ell_2\ra \ell_2},\]
and the multiplicative ergodic theorem of Osceledets \cite{oseledets}, which gives conditions for almost sure (deterministic) values for all the Lyapunov exponents. Many more recent works characterize the Lyapunov exponents under more specific assumptions, most notably for matrices which are rotationally invariant or which have entries that are real or complex Gaussians, see e.g. \cite{gaussian_RMT_6,gaussian_RMT_1,gaussian_RMT_2,gaussian_RMT_3,gaussian_RMT_4,gaussian_RMT_5} as well as the survey \cite{RMT_product_survey} and references therein. 

Case (b), where $d$ is fixed and $n \to \infty$, falls into the setting of free probability. Indeed, one of the great successes of free probability is the idea of ``asymptotic freeness'': in the limit $n \to \infty$, a collection of $d$ independent $n \times n$ random matrices behave like a collection of $d$ freely independent random variables on a non-commutative probability space (see e.g. \cite{AGZ} Chapter 5 or \cite{mingo_speicher} Chapter 1 and 4).  Therefore, case (b) is closely related to a product of $d$ freely independent random variables;  precise results are obtained in \cite{d_fixed_1}. Earlier results \cite{d_fixed_3,d_fixed_2} examine case (b) without explicit use of free probability. The problem of first taking $n \to \infty$ and afterwards taking $d \to \infty$ can also be handled using the tools of free probability in the case of Gaussian matrices, see \cite{tucci}.

As explained in the Introduction and in Section \ref{S:scaling-limits}, the regimes (a), (b) are asymptotically incompatible in the sense that the limits $d\to \infty$, and $n\to \infty$ do not generally commute on the level of the local behavior of the singular value distribution. Indeed, the problem of understanding what happens when both are scaled simultaneously is mentioned as an open problem in \cite{gaussian_RMT_6}. To explain this further, we note that the work of Newman \cite[Thm. 1]{gaussian_RMT_7} in regime (a) shows that when $p=1$ and $n_j\equiv n$ is fixed, the density of the Lyapunov exponents of $\mat{M}^{(d)}$ converges in the limit when first $d\to \infty$ and then $n\to \infty$ to the triangular density 
\[h(\la)=
\begin{cases}
  2\la,&\quad 0<\la<1\\
0,&\quad \text{otherwise}
\end{cases}.
\]
The work of Tucci \cite[Thm. 3.2, Ex. 3.4]{tucci} shows that for Gaussian ensembles related to $\mat{M}^{(d)}$ one obtains the same global limit in the regime (b) when first $n\to \infty$ and then $d\to \infty.$ However, as explained in \cite{gaussian_RMT_6} Section 5, while the global density of all the Lyapunov exponents is the triangular law in both cases, the local behavior (e.g. the fluctuations of the top Lyapunov exponent) is observed to be different depending on the order of the limits even in the exactly solvable special case of products of complex Ginibre matrices. 

From this spectral point of view, Theorem \ref{T:main} gives information about certain averages of the Lyapunov exponents. To see this, fix $n_0$ and let $n_i,\, i\geq 1,$ and $d$ tend to infinity in accordance with \eqref{E:scaling-limit}. Note we we specifically do not take $n_0$ to infinity. Denote by $\si_1, \ldots, \si_{n_0}$ the non-zero singular values of $\mat{M}^{(d)}$, and by $\vec{v}_1,\ldots,\vec{v}_{n_0}$ the corresponding left-singular vectors. Then we have
\[||\mat{M}^{(d)}\vec{u}||^2 ~=~ \sum_{j=1}^{n_0} \si_j^2 \,\langle \vec{u},\, \vec{v}_j\rangle^2.\]
In many situations of interest we can expect that the inner products satisfy $\langle u,\, v_j\rangle^2~\approx~\frac{1}{n_0}$. This happens for example if the vector $\vec{u}$ is chosen uniformly at random on the $n_0$-sphere or when $\vec{u}$ is a fixed vector and the matrix $\mat{W}^{(1)}$ is invariant under right multiplication by an orthogonal matrix. In this setting
\[\log(||\mat{M}^{(d)}\vec{u}||^2)~\approx~\log\lr{\frac{1}{n_0}\sum_{j=1}^{n_0} \si^2_j}.\] 
 %$\langle u,\, v_j\rangle^2~=~\frac{1}{n_0}+O(n_0^{-2})$. This happens for example if the vector $\vec{u}$ is chosen uniformly at random on the $n_0$-sphere or when $\vec{u}$ is a fixed vector and the matrix $\mat{W}^{(1)}$ is invariant under right multiplication by an orthogonal matrix. In this setting
%\[\log(||\mat{M}^{(d)}\vec{u}||^2)~=~\log\lr{\frac{1}{n_0}\sum_{j=1}^{n_0} \si^2_j + O(n_0^{-1})}.\] 
%\approx ~ \mathcal N(-\beta/2, \beta).\]
Hence, Theorem \ref{T:main} can be interpreted as the statement that the logarithm of the average of the non-zero singular values for $||\mat{M}^{(d)}\vec{u}||^2$ is a Gaussian with mean $-\beta/2$ and variance $\beta$ in the limit \eqref{E:scaling-limit}. These non-trivial corrections in $\beta$ can be seen as a finite temperature correction to the maximal entropy regime of Tucci \cite{tucci} in which first $n\to \infty$ and then $d\to \infty$. For more on this point of view, we refer the reader to \cite[Section 3.2]{gaussian_RMT_6}. 

%Finally note that in the limit $d \to \infty$, the singular values are expected to scale like the Lyapunov exponents $\si_j \approx \exp{d \la_j}$.

%Note that our main result Theorem \ref{T:main} indeed gives some information about the local fluctuations of Lyapunov (rather than their overall global position of all the Lyapunov exponents). Indeed, if $n_0$ is fixed, and $d \to \infty$, then the top Lyapunov exponent can be recovered from the limit $\la_{\max} = \lim_{d \to \infty} \frac{1}{d} \log{\norm{ M\vec{u} }^2}$ for almost every fixed vector $u$. Thus, one might expect for large $d$ and $n$ that     

%$\clubsuit\clubsuit\clubsuit\clubsuit$ FILL IN FROM DEIFT $\clubsuit\clubsuit\clubsuit$

Finally, in the specific case where the random matrices $\mat{X}^{(j)}$ are complex Ginibre matrices (i.e. the matrix entries are iid complex Gaussian), very recent work \cite{dsl_2,joint_limit} looks at the limiting spectrum under the joint scaling limit $d\to \infty$, $n\to \infty$ where the ratio $d/n$ is fixed or going to $\infty$. This work analyzes exact determinental formulas for the joint distribution of singular values available in the case of complex Ginibre matrices. The analogous formulas for real Gaussian matrices given in \cite{gaussian_RMT_5} are significantly more complicated and such an explicit analysis appears to be much more difficult.

\subsection{Contribution of the Present Work}\label{S:contribution} In the context of these previous random matrix results, let us point out four novel aspects of Theorem \ref{T:main}. First, it deals with the joint $d,n\to\infty$ limit for a large class of non-Gaussian matrices with real entries. There is no integrable structure to our matrix ensembles, and we rely instead on a sum-over-paths approach to analyze the moments \eqref{E:Mu-moments} and a martingale CLT approach for obtaining the KS distance estimates \eqref{E:KS-dist}. 

Second, the ensembles in Theorem \ref{T:main} include the somewhat unusual diagonal $\mathrm{Bernoulli}(p)$ matrices $\mat{D}$ as part of model. Our original motivation for including these is the connection to neural networks explained in Section \ref{S:neural-nets}. In essence, the matrices $\mat{D}^{(i)}$ can be interpreted as adding iid $\{0,1\}-$valued spins to the usual sum over paths approach to moments of products of matrices. Only ``open'' paths that have spin $1$ on every vertex contribute to the sum, causing open paths to be correlated. Previously, Forrester \cite{gaussian_RMT_2} and Tucci \cite{tucci} considered the case when $\mat{D}^{(i)}$ were \textit{deterministic} positive definite matrices. 

An additional novelty of Theorem \ref{T:main} is it proves the distribution of $ \big|\big|\mat{M}^{(d)}\vec{u}\big|\big|_2^2$ is (mostly) universal: it does not depend on the higher moments of the distribution $\mu$ beyond the mean and variance, with the exception of the fourth moment $\mu_4$ appearing in $\beta$ in the term $\norm{\vec{u}}_4^4 (\mu_4 - 3)/pn_1$. In the regime $n_j\equiv n$ and $d/n\equiv \beta$, this term is a $1/n$ correction.

The fourth and final novelty of Theorem \ref{T:main} we would like to emphasize is that our results are non-asymptotic, i.e. we obtain an explicit error term of the form $\sum_{i=1}^d {n^{-2}_i}$. This is particularly useful when using Theorem \ref{T:main} for studying gradients in randomly initialized neural networks (see Section \ref{S:neural-nets}).

Finally, we remark that Theorem \ref{T:main} only studies $ \big|\big|\mat{M}^{(d)}\vec{u}\big|\big|_2^2$ for a fixed vector $\vec{u}$, and therefore leaves several questions open: for instance the joint law of $\{ \big|\big|\mat{M}^{(d)}\vec{u}^{(1)}\big|\big|_2^2, \ldots, \big|\big|\mat{M}^{(d)}\vec{u}^{(\ell)}\big|\big|_2^2 \}$ for a list of vectors $\{ \vec{u}^{(1)},\ld,\vec{u}^{(\ell)} \}$ and more generally the limiting spectral distribution of the matrices $\mat{M}^{(d)}$. We plan to address these questions in forthcoming work.

\subsection{Connection to Random Polymers}\label{S:RP}
The matrix ensembles $\mat{M}^{(d)}$ studied in this article, in the case $n_i=n, \,\,i=1,\ld,d$, are related to  directed random polymers on the complete graph of size $n$. This model were recently explored in detail (c.f. e.g. \cite{random_polymers}). A key object for these polymers is the line-to-line partition function
\[ 
Z(d) =  \sum_{\pi \in \{1,\ld,n\}^d } \exp\left( \frac{1}{T} \et(i,\pi(i)) \right),
\]
where $T$ is the temperature of the model, and $\left\{ \et(i,j) \right\}_{(i\in \bN,1\leq j\leq n)}$ are i.i.d. mean zero random variables that make up the underlying disordered environment. When the sum over $\{1,\ld,n\}^d$ is written via products of $d$ matrices of size $n\times n$, the disordered environment can be viewed as a multipartite graph made of $d$ vertex clusters $V_1,\ld,V_d$ of size $n$ with (directed) edges from all vertices in $V_i$ to all vertices in $V_{i+1}.$ The edges of this graph are then decorated with the corresponding matrix entries $\exp\lr{ \frac{1}{T} \et(i,j) },$ which are strictly positive, making $Z(d)$ a sum over paths from the input to the output of this graph. Each path is weighted by its energy, given by the product of weights along the path. 

The fact that the weights are positive makes the analysis of the partition function of this traditional random polymer model different than the analysis of the matrix product $\mat{M}^{(d)}$ defined in \eqref{E:J-def}. In particular, no cancellation is possible between the terms in the definition of $Z(d)$ above, causing $Z(d)$ to be exponential in $d$. The $n$ fixed and $d \to \infty$ limit of of the partition function in the case of these positive weights is the object of study in \cite{random_polymers}. As explained in Section \ref{S:contribution}, Theorem \ref{T:main} studies a different regime where both $d \to \infty, n \to \infty$ at the same time. The fact that our weights are mean zero, gives rise to significant cancellation in the terms of $Z_d(\vec{u})$ from \eqref{E:F-def}, so that the partition function in our mean zero model \textit{does not grow exponentially with $d$} provided $n$ grows with $d$ as in \eqref{E:scaling-limit}. Additionally, if $p < 1$ in our model, the effect of the diagonal Bernoulli matrices $\mat{D}^{(i)}$ is to close every vertex with probability $1-p$. The sum over paths in our partition function then becomes the sum only over those paths that pass through vertices that are open.

\subsection{The Case $p=1/2$ as Gradients in Random Neural Nets}\label{S:neural-nets}One of our motivations for studying the ensembles $\mat{M}^{(d)}$ is that, as we prove in Proposition \ref{P:even-grads} below, the case $p=1/2$ corresponds exactly to the input-output Jacobian in randomly initialized neural networks with $\Relu$ activations. To explain this connection, fix $d, n_0,\ldots, n_d\in \mathbb N$. A neural network with $\Relu$ activations, depth $d$, and layer widths $n_0,\ldots, n_d$ is any function of the form
\[\mathcal N(x)=\Relu \circ A^{(d)}\circ \cdots \circ \Relu \circ A^{(1)}(x),\]
where $A^{(j)}:\R^{n_{j-1}}\ra \R^{n_j}$ are affine
\[A^{(i)}(\vec{x})=\mat{W}^{(i)}\vec{x}+\vec{B}^{(i)},\qquad \mat{W}^{(i)}\in \mathrm{Mat}\lr{n_{i}, n_{i-1}},\quad \vec{B}^{(i)}\in \R^{n_i},\]
and for $m\geq 1$ and any vector $\vec{v}=\lr{v_1,\ldots, v_m}\in \R^m$ we write
\[\Relu\lr{\vec{v}}= \lr{\max\set{0, v_1},\ldots, \max\set{0,v_m}} \in \R^m.\]
The matrices $\mat{W}^{(j)}$ and vectors $\vec{B}^{(j)}$ are called, respectively, the weights and biases of $\mathcal N$ at layer $j,$ while $d,n_0,\ldots, n_d$ collectively define the architecture of $\mathcal N$. We will write $\Act^{(0)}\in \R^{n_0}$ for an input to $\mathcal N$ and will define 
\[\Act^{(j)}~\defequal~\Relu \lr{\act^{(j)}},\qquad \act^{(j)}~\defequal~ A^{(j)}\circ \Relu\circ \cdots \circ \Relu \circ A^{(1)}(\Act^{(0)})\]
to be the vectors of activities before and after applying $\Relu$ at the neurons at layer $j$. \\

In practice, the weights and biases in a neural network are first randomly initialized and then optimized by (stochastic) gradient descent on a task-specific loss $\mathcal L=\mathcal L(\Act^{(d)})$ that depends only on the outputs of $\mathcal N.$ A single gradient descent update for a trainable parameter $W$ (i.e. an entry in one of the weight matrices $\mat{W}^{(i)}$) is 
\begin{equation}
W\quad \longleftarrow \quad W ~-~\lambda \,\frac{\partial \mathcal L}{\partial W},\label{E:param-update}
\end{equation}
where $\lambda >0$ is the learning rate. An important practical impediment to gradient based learning is the \textit{exploding and vanishing gradient problem} (EVGP), which occurs when gradients are numerically unstable:
\[\mathrm{EVGP}\quad \longleftrightarrow\quad \abs{\frac{\partial \mathcal L}{\partial W}}\approx 0~~\text{or}~~\infty,\]
making the parameter update \eqref{E:param-update} too small to be meaningful or too large to be precise. An important intuition is that the EVGP will be most pronounced at the start of training, when the weights and biases of $\mathcal N$ are random and the implicit structure of the data being processed has not yet regularized the function computed by $\mathcal N$. 

As explained below, the EVGP for a depth $d$ $\Relu$ net $\mathcal N$ with hidden layer widths $n_0,\ld,n_d$ is essentially equivalent to having large fluctuations for the entries (or, in the worst case, for the singular values) of the Jacobians of the transformations between various layers:
\begin{equation}\label{E:Jac-def}
 \Jac^{(j\ra j')}~\defequal~ \lr{ \partial \Act_a^{(j')}\big/\partial \Act_b^{(j)}}_{\substack{1\leq a \leq n_{j'}\\1\leq b \leq n_j}},\qquad 0\leq j < j' \leq d.
\end{equation}
The next result shows the singular value distribution of $\mat{M}^{(j'-j)}$ is that same as that of $\Jac^{(j\ra j')}$ since
\[\lr{\Jac^{(j\ra j')}}^T\Jac^{(j\ra j')}~\dequal~\lr{\mat{M}^{(j'-j)}}^T\mat{M}^{(j'-j)}.\]
Proposition \ref{P:even-grads} also shows that, for any collection of vectors $\vec{u}_1,\ld,\vec{u}_k \in \R^{n_j},$ we have the following equality in distribution when $p=1/2$:
\[\left\{\norm{\Jac^{(j\ra j')}\vec{u}_i}_2^2,\quad i=1,\ld,k \right\}~~~\dequal~~~\left\{\norm{\mat{M}^{(j'-j)}\vec{u}_i}_2^2,\quad i=1,\ld,k \right\}.\]
\begin{prop}\label{P:even-grads}
Let $\mathcal N$ be a $\Relu$ net with depth $d$ and layer widths $n_0,\ldots, n_d$. Fix $0\leq j<j'\leq d.$ Suppose the weights of $\mathcal N$ are $\mat{W}^{(i)}$, which are drawn iid from the measure $\mu$ as in the original definition \eqref{E:model-def}. Then, writing $\eta^{(j')} \in \vec{\bR^{n_{j'}}}$ for the $n_{j'}$ dimensional ${\pm1}$-Bernoulli random vector, whose entries are independent and take the values $\pm 1$ with probability $1/2$, we have
\[
\mathrm{Diag}(\eta^{(j')})\Jac^{(j\ra j')}~~\dequal~~ \mat{D}^{(j')}\mat{W}^{(j')} \cdots \mat{D}^{(j)}\mat{W}^{(j)} ,\qquad j'>j,
\]
where $\mat{D}^{(i)}$ are diagonal $\{0,1\}$-Bernoulli matrices as in \eqref{E:model-def} with parameter $p =\half$ and $\dequal$ denotes equality in distribution.
\end{prop}

Before proving Proposition \ref{P:even-grads}, let us explain why the functions $||\Jac^{(j'-j)}\vec{u}||^2_2$ that we study in Theorem \ref{T:main} are related to the EVGP. Due to the compositional nature of the function computed by $\mathcal N,$ we may use the chain rule to write, for the weight $W_{a,b}^{(j)}$ connecting neuron $a$ to neuron $b$ in layer $j,$
\[\frac{\partial \mathcal L}{\partial W_{a,b}^{(j)}}=\langle \nabla L(\Act^{(d)}),\,\Jac_b^{(j\ra d)}\rangle\one\{\Act_b^{(j)}>0\}\Act_a^{(j-1)},\qquad {\nabla \mathcal L( \Act_b^{(j)})}=\lr{\partial \mathcal L/\partial \Act_q^{(d)},\,\, q=1,\ld, n_d}\]
where $\Jac_b^{(j\ra d)}$ is the $b^{th}$ column of $\Jac^{(j\ra d)}.$ Therefore, fluctuations of the gradient descent update $\partial \mathcal L/\partial W_{a,b}^{(j)}$ are captured precisely by fluctations of bi-linear functionals $\langle{\Jac^{(j\ra j')}\vec{u},\vec{v}\rangle}$ of various layer to layer Jacobians in $\mathcal N$. We study in this article $\vec{u}=\vec{v}$ and obtain in Theorem \ref{T:main} precise distribution and moment estimates on these quantities. For instance, Theorem \ref{T:main} combined with Proposition \ref{P:even-grads} immediately yields the following 
\begin{cor}\label{C:EVGP}
  Let $\mathcal N$ be a fully connected depth $d$ $\Relu$ net with hidden layer width $n_0,\ld,n_d$ and  randomly initialized weights drawn i.i.d. from the measure $\mu$ and scaled to have variance $2/\text{fan-in}=2/n_i$ as in \eqref{E:model-def}. Suppose also that the biases of $\mathcal N$ are drawn i.i.d. from any measure satisfying same assumptions as the measure $\mu.$ Fix any $\vec{u}\in \R^{n_0}$ with $\norm{\vec{u}}=1$ and write $\Jac^{(d)}$ for the input-output Jacobian of $\mathcal N$. We have,
\[\sup_{-\infty<t<T<\infty}\abs{\p\lr{t<\beta^{-1/2}\lr{\log\lr{\frac{n_d}{n_0}\norm{\Jac^{(d)}\vec{u}}^2}+\frac{\beta}{2}}<T}~-~\int_{t}^T e^{-s^2/2}\frac{ds}{\sqrt{2\pi}}}~=~O\lr{\sum_{i=1}^d \frac{1}{n_i^2}},\]
where 
\[\beta ~\defequal~ 5 \sum_{i=1}^{d} \frac{1}{n_i} +\frac{2(\mu_4-3)}{n_1}\norm{\vec{u}}_4^4 \]
and the implicit constant is uniform when $\beta$ ranges over a compact subset of $(0,\infty)$. 
\end{cor}

 \noindent For more information about the EVGP, statistics of gradients in random $\Relu$ nets, and distribution of the singular values of the input-output Jacobian we refer the interested reader to \cite{hanin2018neural, pennington2017resurrecting, PenningtonSG18, pennington2017nonlinear} for more details.

%\begin{cor}
 % Let $\mathcal N$ be a $\Relu$ net with depth $d$ and layer widths $n_0,\ldots, n_d$. Suppose the weights of $\mathcal N$ are given by $W^{(j)}$ as in \eqref{E:Wp-def} and the biases of $\mathcal N$ are chosen indep Then
%\[\norm{J^{(d)}u}^2....\]
%\end{cor}

\section{Proof of Proposition \ref{P:even-grads}}\label{S:even-grads-pf}

\subsection{Idea behind the proof}
The essential idea behind Proposition \ref{P:even-grads} is to notice that the derivative of the $\Relu$ function is $\Relu^{\prime}(x) = \one\{x > 0\}$, so when doing the chain rule to compute $\Jac^{(d)}$, we find the following $\{0,1\}$-valued diagonal matrices naturally appearing:
\begin{equation}
\mathrm{Diag}\left( \one\{ \mat{W}^{(i)} \vec{x} + \vec{B}^{(i)} \} \right) \label{E:diag-matrix}
\end{equation}
Since the random weights $W^{(i)}_{a,b}$ and biases $B^{(i)}_a$ are symmetrically distributed around $0$ (i.e. $-W^{(i)}_{a,b} \dequal W^{(i)}_{a,b}$) and have no atoms, it is easily verified that each entry in $\mat{W}^{(i)} \vec{x} + \vec{B}^{(i)}$ is equally likely to be positive or negative regardless of the value of $\vec{x}$. Hence the matrix in equation \eqref{E:diag-matrix} is equal in distribution to the Bernoulli matrix $\mat{D}^{(i)}$ when $p = \half$. This informally explains the connection between $\Jac^{(d)}$ and $\mat{M}^{(d)}$.

It remains to see that these diagonal matrices are independent of each other (since the outputs of the previous layer are fed into to subsequent layers, so are not a priori independent). This again will be a consequence of the fact the underlying random variables are symmetrically distributed, and will be formally verified by conjugating the weights and biases of the network by random $\pm 1$ random variables. This doesn't change the distribution of the network, but will allow us to see the independence between layers in a more concrete way. 

\subsection{Proof of Proposition \ref{P:even-grads}}

\begin{proof}[Proof of Proposition \ref{P:even-grads}] Fix a neural net $\mathcal N$ as in the statement of proposition \ref{P:even-grads} and denote its weights and biases at layer $j$ by $W_{a,b}^{(j)}$ and $B_a^{(j)}$. For each $j,$ let 
\[\set{\xi_{a}^{(j)},\eta_{a}^{(j)},\,\, j=1,\ldots, d,\,\, a=1,\ldots, n_{j-1}}\] 
be an i.i.d. collection of random variables that each take values $\pm 1$ with probability $1/2$. We will also define
\[\eta_a^{(0)}=1,\quad \forall a =1,\ldots, n_0.\]
Consider the neural net $\widehat{\mathcal N}$ with weights $\widehat{W}_{a,b}^{(j)}$ and biases $\widehat{B}_a^{(j)}$ defined by changing the signs of the weights and biases of $\mathcal N$ as follows:
\[
\widehat{W}_{a,b}^{(j)}\defequal\xi_{a}^{(j)}W_{a,b}^{(j)}\et_{b}^{(j-1)},\ \widehat{B}_{a}^{(j)}\defequal\xi_{a}^{(j)}B_{a}^{(j)}
\]
so that 
\[\mat{\widehat{W}}^{(j)}=\mathrm{Diag}(\xi^{(j)}) \,\mat{W}^{(j)}\, \mathrm{Diag}(\et^{(j-1)}),\qquad \widehat{\vec{B}}^{(j)}=\mathrm{Diag}(\xi^{(j)}) \,\vec{B}^{(j)}.\]
We will denote by $\widehat{\Act}^{(j)},\, \widehat{\Jac}^{(d)}$ the activations and input-output Jacobian for $\widehat{\mathcal N},$ both computed at the same fixed input
\[\widehat{\Act}^{(0)} \defequal \Act^{(0)}.\]
Note that since we've assumed that $\mat{W}^{(j)},\vec{B}^{(j)}$ have distributions that are symmetric around $0$, we have \[\{\widehat{W}^{(j)},\widehat{B}^{(j)},\,j=1,\ldots,d\}\dequal\{W^{(j)},B^{(j)},\,j=1,\ldots, d\}.\]
Hence, since the weights of the two networks are identically distributed,
\[ \mathrm{Diag}(\eta^{(d)})\widehat{\Jac}^{(d)}\dequal \mathrm{Diag}(\eta^{(d)})\Jac^{(d)}.\]
On the other hand, the chain rule yields the following recursion for $\mathrm{Diag}(\eta^{(d)})\widehat{\Jac}^{(d)}:$
\begin{align}
\mathrm{Diag}(\eta^{(d)})\widehat{J}^{(d)}&=\mathrm{Diag}(\eta^{(d)})\widehat{\mat{D}}^{(d)}\widehat{\mat{W}}^{(d)}\cdots \widehat{\mat{D}}^{(1)}\widehat{\mat{W}}^{(1)} \nonumber \\
&=\mathrm{Diag}(\eta^{(d)})\widehat{\mat{D}}^{(d)}\mathrm{Diag}(\xi^{(d)})\mat{W}^{(d)}\mathrm{Diag}(\eta^{(d-1)})\cdots \widehat{\mat{D}}^{(1)}\mathrm{Diag}(\xi^{(1)})\mat{W}^{(1)}\mathrm{Diag}(\eta^{(0)})\nonumber \\
&=\widehat{\mat{D}}^{(d)}\sigma^{(d)}W^{(d)}\,\mathrm{Diag}(\eta^{(d-1)})\widehat{J}^{(d-1)}, \label{E:J-recur}
\end{align}
where 
\[\widehat{\mat{D}}^{(j)} \defequal \mathrm{Diag}\lr{\one\left\{\widehat{\act}_a^{(j)}>0\right\},\,\, a=1,\ldots, n_j},\quad \sigma^{(j)} \defequal \mathrm{Diag}(\eta^{(j)})\mathrm{Diag}(\xi^{(j)})\]
and we've used that diagonal matrices commute. Note that apriori the matrices $\widehat{\mat{D}}^{(j)}$ depend on the weights and biases $\widehat{\mat{W}}^{(i)},\widehat{\vec{B}}^{(i)}$ for $i\leq j$ since $\widehat{\act}^{(j)}=\widehat{\mat{W}}^{(j)}\widehat{\Act}^{(j-1)}+\widehat{\vec{B}}^{(j)}$. However, we will now verify the following claim about the collection of matrices  $\widehat{\mat{D}}^{(j)}$ and variables $\sigma^{(j)}$:
\begin{equation}
\set{\widehat{\mat{D}}^{(j)},\, \sigma^{(j)},\,\,j=1,\ldots, d}\quad \text{is independent of}\quad \set{\mat{W}^{(j)},\vec{B}^{(j)},\,\, j=1,\ldots, d}\label{E:prop-goal}
\end{equation}
and that moreover the collection $\set{\widehat{\mat{D}}^{(j)},\,j=1,\ldots, d}$ is independent and that each $\widehat{\mat{D}}^{(j)}$ is distributed like a diagonal matrix with independent diagonal entries taking the values of $\{0,1\}$ with probability $1/2.$ Once we have proven this, since $\sigma^{(j)}\mat{W}^{(j)} \dequal \mat{W}^{(j)}$, then equation \eqref{E:J-recur} shows that $\mathrm{Diag}(\eta^{(d)})\widehat{\Jac}^{(d)} \dequal \widehat{\mat{D}}^{(d)} \mat{W}^{(d)} \mathrm{Diag}(\eta^{(d-1)})\widehat{\Jac}^{(d-1)}$ and $\widehat{\mat{D}}^{(d)} \dequal \mathrm{Diag}(\xi_1,\ld,\xi_{n_d})$, a diagonal $\{0,1\}$-Bernoulli random variables independent of everything else. This will complete the proof of the present proposition since this is exactly the recurrence for the matrices $M^{(d)}$. 

To prove \eqref{E:prop-goal}, we will use the fact that two random variables $X,Y$ are independent if the distribution of $X$ given $Y=y$ does not depend on the value of $y.$ 
%the following observation. 
%\begin{lem}
 % Suppose $X$ is a random variable that takes values in a finite set and $Y$ is any random variable that is defined on the same probability space as $X.$ If the conditional distribution of $X$ given $Y$ is uniform on $S$, then $X$ is independent of $Y$ with a marginal distribution that is uniform on $S.$
%\end{lem}
%\begin{proof}
 % For any $x\in S$ and any measurable set $A$, we have
%\[\P\lr{X=x,\, Y\in A}=\int_A \P\lr{X=x~|~Y=y}d\mathbb P_Y(y)=\frac{1}{\abs{S}}\P(Y\in A).\]
%Taking $A$ to be a full measure set, we find that $\P\lr{X=x}=1/\abs{S}$ and hence that $\P\lr{X=x,\, Y\in A}=\P(X=x)\P(Y\in A),$ as desired. 
%\end{proof}
That is, \eqref{E:prop-goal} will follow once we show that for any fixed sequences $s_a^{(j)}, r_a^{(j)}\in \set{\pm 1}$ that
\[
\p\left(\bigcap_{a,j}\left\{ \widehat{D}_{a}^{(j)}=s_{a}^{(j)},\si_{a}^{(j)}=r_{a}^{(j)}\right\} \given{ \mat{W}^{(i)},\vec{B}^{(i)},\,\, i=1,\ldots,d }\right)=\prod_{i=1}^{d}\left(\frac{1}{2^{n_{i}}}\right)^{2}.
\]
To check this equality, it suffices to show that given $\set{s_a^{(j)}, \, r_a^{(j)}}$ there is exactly one possible configuration for the variables $\xi_{a}^{(j)},\et_{a}^{(j)}$ for which the event $\cap_{a,j}\left\{\widehat{D}_{a}^{(j)}=s_{a}^{(j)},\si_{a}^{(j)}=r_{a}^{(j)}\right\} $ occurs. The resulting probability then follows since $\xi_{a}^{(j)},\et_{a}^{(j)}$ are i.i.d. variables that each take the values $\pm 1$ with probability $1/2.$ The proof is by induction: we will show that for each $i=1,\ldots, d$, given $W^{(j)},B^{(j)},\,j\leq i$ there is a unique configuration for the variables $\xi_{a}^{(j)},\et_{a}^{(j)},\,\, j\leq i$ that leads to the event $\cap_{a,j\leq i}\left\{ \widehat{D}_a^{(j)}=s_{a}^{(j)},\si_{a}^{(j)}=r_{a}^{(j)}\right\} $. When $i=1,$ we have
\[\widehat{D}_a^{(1)}=\one\left\{\sum_{b=1}^{n_0} \widehat{W}_{ab}^{(1)}\widehat{\Act}^{(0)}_b + \widehat{B}_a^{(1)} >0\right\}=\one\left\{\xi_a^{(1)}\lr{\sum_{b=1}^{n_0} W_{ab}^{(1)}\eta_b^{(0)}\widehat{\Act}^{(0)}_b + B_a^{(1)}} >0\right\}.\]
Recalling that $\eta_b^{(0)}=1$ for all $b$, we see that for each $a,$ there is a unique value of $\xi_a^{(1)}$ for which 
\[\widehat{D}_a^{(1)}=s_a^{(1)}.\]
Then, for this value of $\xi_a^{(1)},$ since we have $\si_a^{(1)}=\xi_a^{(1)}\eta_a^{(1)}$, there is a unique value of $\eta_a^{(1)}$ for which $\si_a^{(1)}=r_a^{(1)}.$ The proof of the inductive step is identical. Namely, suppose we have determined the values of $\xi_a^{(j)},\eta_a^{(j)}$ for $j\leq i.$ Then, given the weights and biases $W^{(j)},B^{(j)},\,\, j\leq i,$ we have uniquely determined $\widehat{\Act}^{(j)}.$ Then, given this value for $\widehat{\Act}^{(j)}$, for every $a=1,\ldots, n_{i+1},$ there is a unique value $\xi_a^{(i+1)}$ for which $\widehat{D}_a^{(i+1)}=s_a^{(i+1)}.$ And finally, given this value for $\xi_a^{(i+1)}$ there is a unique value of $\eta_a^{(i+1)}$ so that $\si_a^{(i+1)}=r_a^{(i+1)}.$ This completes the proof. \end{proof}

\section{Proof of Theorem \ref{T:main}: Moment Estimates and Path Counting} \label{S:moments}

\subsection{Outline of Proof of Equation \eqref{E:Mu-moments}} We begin by indicating the general plan for the proof of Equation \eqref{E:Mu-moments} from Theorem \ref{T:main}, which consists of two steps. First, in Proposition \ref{prop:2k_tO_reduction} below, we express the expectation in \eqref{E:Mu-moments} as a sum over $k$-tuples of paths $V\in[n_{0}]^{k}\times\cdots\times[n_{d}]^{k}$. The precise result is the following
\begin{prop}[Moments of $||\mat{M}\vec{u}||^2$ as a sum over paths]\label{prop:2k_tO_reduction}
With the notation of Theorem \ref{T:main}, for each $k$ we have
\begin{equation}
\e\left[\norm{\mat{M}\vec{u}}^{2k}\right]  =\left(\prod_{i=0}^{d-1}\frac{1}{n_{i}^{k}}\right)\sum_{V\in[n_{0}]^{k}\times\ld\times[n_{d}]^{k}} u^2_{V(0)} \prod_{i=1}^{d}C(V(i-1),V(i)),\label{E:Mu-k-paths}
\end{equation}
where $u_{V(0)} \defequal \prod_{j=1}^{k} u_{V_j(0)}$, and $C$ is defined by:
\begin{equation}
  \label{E:C-hat-def-again}
C(V(i-1),V(i)) \defequal wt\left(2m_{V(i-1),V(i)}\right) \frac{c_{2k}(2m_{V(i-1),V(i)})}{c_{k}(m_{V(i-1),V(i)}) } p^{\#V(i)-k}
\end{equation}
with $\#v$ denoting the number of unique entries in a tuple $v\in[n]^{k}$, $m_{x,y}$ being the multiplicity of edges appearing in the set $\left\{ (x_{i},y_{i})\right\} _{i=1}^{k}$ as in \eqref{E:mxy-def}, and $c_\ell$ a combinatorial factor given by \eqref{E:c-def}, and $wt$, defined in \eqref{E:wt-def} denoting a weight function that depends on the moments of the entries of the weights matrices $\mat{W}^{(i)}$. 
\end{prop}
Note that the definition of $C$ in equation \eqref{E:C-hat-def-again} depends only on the collection of vertices $V(i-1)\in [n_{i-1}]^k$ and $V(i)\in [n_i]^k$, the moments of the measure $\mu$ according to which the entries of the matrices $\underline{W}^{(i)}$ are distributed, and the parameter $p.$ The utility of equation \eqref{E:Mu-k-paths} is that it is written as a product over this function of adjacent layers (rather than the whole path $V$), which will make it much easier to analyze. 

The next step in the proof of the moment estimate \eqref{E:Mu-moments} is to obtain upper and lower bounds for the expression in \eqref{E:Mu-k-paths} that match up to corrections of size $\sum_{i=1}^d n_i^{-2}.$ This is done in Section \ref{S:proof-completion}. The main idea here is to treat the sum \eqref{E:Mu-k-paths} as an expectation where each $V(i) \in [n_i]^k$, $i \geq 1$ is chosen independently according to the uniform distribution on $[n_i]^k$. The leading term in this expectation comes from  event that the entries $\left\{ V_1(i), \ld V_{k}(i) \right\}$ are all distinct, which happens in layer $i$ with probability $1 - O(n^{-1}_i)$. When this happens, $C(V(i-1),V(i))=1.$ The subleading term comes from the event that $\left\{V_1(i), \ld, V_k(i)\right\}$ has exactly one element that appears twice, with the others distinct. In each layer, the probability of this type of ``collision'' is $\binom{k}{2} n_i^{-1},$ and $C(V(i-1),V(i))$ typically contributes $3/p$ when this happens. Hence, heuristically speaking, we have
\[\e\left[\frac{n_{0}^k}{n_{d}^{k}}\norm{\mat{M}\vec{u}}^{2k}\right] \approx \cE\left[ \prod_{i=1}^d \left(\frac{3}{p}\right)^{\one\{V(i) \text{ has a ``collision''}\}} \right] \approx \prod_{i=1}^d \left(1+\left(\frac{3}{p}-1\right)\binom{k}{2}\frac{1}{n_i}\right).\]
This is almost correct, except at the first layer, where the vector $\vec{u}$ acts as special initial condition and slightly deforms the term in this product when $i=1.$ Section \ref{S:proof-completion} makes this argument precise.

\subsection{Edge Sets, Multiplicities, and Paths}

In this section we develop some notation and basic results which is used to clarify the ``path counting'' needed to prove Proposition \ref{prop:2k_tO_reduction} below. The major result that is developed in this section, and is needed for Proposition \ref{prop:2k_tO_reduction}, is the enumeration the set of paths in Lemma \ref{lem:edge_sequence_path_counting}.  We will use the following notation
conventions:
\begin{itemize}
\item $n,n^{\prime}$ denote natural numbers $\in\bN$ 
\item $[n]~\defequal~\left\{ 1,2\ld,n\right\} $
\item $[n]^{\ell}~\defequal~\left\{ (x_{1},\ld,x_{\ell}):\ x_{j}\in[n]\ \forall1\leq j\leq\ell\right\} $
\end{itemize}
For $n,n^{\pr}\in\bN$, we will denote by 
\[\Si^{\ell}(n,n^{\prime})~\defequal~\left\{\{e_1,\ld ,e_\ell\}~|~e_j\in [n]\times [n']\right\} \]  
the collection of all unordered sets of $\ell$ directed edges in the complete bipartite graph of $K_{n,n^{\prime}}$, which we think of as a directed graph with edges going from $[n]$ to $[n^{\prime}]$. Note that some edges may appear multiple times: we consider them with multiplicity, thinking of $E\in\Si^{\ell}(n,n^{\prime})$ as a multi-set (e.g. the directed edge $(1,1)$ can appear twice in $E$). To every edge set $E\in\Si^{\ell}(n,n^{\prime})$ we will associate the \textbf{edge multiplicity}, $m_{E}\in\bN^{n\times n^{\prime}},$ by:
\[
m_{E}(a,b)~\defequal~\text{number of times the edge }(a,b)\text{ appears in }E.
\]
We will also use the notation: 
\begin{align*}
m_{E}(\ast,b)~\defequal~\sum_{a=1}^{n}m_{E}(a,b) & ~=~\text{number of times }b\text{ appears as a right endpoint of edges in }E\\
m_{E}(a,\ast)~\defequal~\sum_{b=1}^{n^{\prime}}m_{E}(a,b) & ~=~\text{number of times }a\text{ appears as a left endpoint of edges in }E
\end{align*}
Every edge set $E\in\Si^{\ell}(n,n^{\prime})$ is uniquely defined by its multiplicity, and we will often find it more convenient to work with the multiplicities rather than the edge sets directly. 

We will need to need to translate back and forth between $E\in \Si^{\ell}(n,n')$ and the multisets of its left and right endpoints. Specifically, for $E=\{e_j=(a_j,b_j),\, j=1,\ld, \ell\}\in \Si^{\ell}(n,n')$ we define the multisets
\[R(E) \defequal \bigcup_{j=1}^\ell \{b_j\},\qquad L(E) \defequal \bigcup_{j=1}^\ell \{a_j\}\]
of right and left endpoints of $E$ counted with multiplicity. Conversely, given \textit{ordered} sets of left, right endpoints $x\in[n]^{\ell}$, $y\in[n^{\pr}]^{\ell}$, we define the corresponding element of $\Si^\ell(n,n')$ by its multiplicity
\[m_{x,y}~\defequal~ m_{\left\{ \left(x_{i},y_{i}\right)\right\}_{i=1}^{\ell}}, \qquad m_{x,y}(a,b) =\abs{\left\{ i:(x_{i},y_{i})=(a,b)\right\} }.\]
This is the set one gets by drawing an edge between each entry of $x$ and the corresponding entry of $y$ and then forgetting the order in which the edges were drawn but remembering the multiplicity. Note that this map from ordered sets of left and right endpoints is many to one. This will come up in our computations, and to keep track of this, we make the following definition.  
\begin{defn}
\label{def:c_function} Fix some edge set $E\in\Si^{\ell}(n,n^{\prime})$,
with corresponding edge multiplicity $m_{E}\in\bN^{n\times n^{\prime}}$and
some $\ell$-tuple $y\in[n^{\prime}]^{\ell}$ so that as unordered multisets
\[y = R(E).\]
Define:
\begin{equation}
c_{\ell}(m_{E})~\defequal~\abs{\left\{ x\in[n]^{\ell}:\ m_{x,y}=m_{E}\right\} }\label{E:c-def}
\end{equation}
\end{defn}

\begin{lem}
\label{lem:c_is_well_defined} $c_{\ell}(m_{E})$ is well defined. That is, the enumeration depends only on $E$ and not on the choice of $y$ stated in Definition \ref{def:c_function}. Moreover, $c_{\ell}(m_{E})$ has the following explicit formula in terms of multinomial coefficients:
\begin{equation}
c_{\ell}(m_{E})=\prod_{i=1}^{n^{\prime}}\binom{m_{E}(\ast,i)}{m_{E}(1,i),m_{E}(2,i),\ld,m_{E}(n,i)}\label{E:c-formula}
\end{equation}
\end{lem}

\begin{proof}
To see that$\abs{\left\{ x\in[n]^{\ell}:\ m_{x,y}=m_{E}\right\} }$
does not depend on $y$ note that for any $y'\in [n']^\ell$, we have 
\[y'=R(E)=y\]
if and only if $y'=\sigma(y)$ for some $\sigma$ in the symmetric group on $\ell$ elements. Further, for any such $\sigma$, we have
\begin{equation}
m_{x,y}=m_{\sigma(x),\sigma(y)}.\label{E:mxy-def}
\end{equation}
Thus, $x\mapsto \sigma(x)$ is a bijection between
$\left\{ x\in[n]^{\ell}:\ m_{x,y}=m_{E}\right\} $ and $\left\{ x\in[n]^{\ell}:\ m_{x,y^\prime}=m_{E}\right\} $
for any permutation $\si\in S_{\ell}$, proving that $c(m_E)$ is indeed well-defined. To obtain the multinomial coefficient formula for $c(m_{E})$, for each $t \in \bN,\, \ell \in \bN, \, x\in [t]^\ell$ define the set of indices:
\[I_s(x) \defequal \{i\in [\ell]~|~ x_i = s\}, \, s \in [t] \]
and for every $I\subseteq [\ell]$ define for $x = (x_1,\ldots, x_\ell)$ the multiset of entries of $x$:
\[x_I \defequal \{x_i~|~ i \in I\}.\]
With this notation, we have
\begin{equation}
m_{x,y}=m_{E} ~\Leftrightarrow~ \forall j\in [n],\, i\in [n'],\, \abs{I_j(x_{I_i(y)})}=m_E(j,i). \label{E:m_count}
\end{equation}
Thus, enumerating $\abs{\left\{ x\in[n]^{\ell}:\ m_{x,y}=m_{E}\right\} }$ amounts to counting the number of ways the indices of $x$ can be arranged in order to satisfy \eqref{E:m_count}. This is counted by multinomial coefficients, and the formula \eqref{E:c-formula} then follows by standard enumeration principles.
\end{proof}
\noindent Our path counting approach to proving Proposition \ref{prop:2k_tO_reduction}, involves the combinatorics of certain paths decorated by the moments of measure $\mu$ according to which the entries of matrices $\mat{W}^{(i)}$ are drawn. Accordingly, for each $n,n^{\prime}\in\bN$, we associate a \textbf{weight} to an edge multiplicity $m_{E}\in\bN^{n\times n^{\prime}}$ given in terms of the moments of the measure $\mu$ by:
\begin{equation}
wt\left(m_{E}\right)~\defequal~\prod_{a\in[n],b\in[n^{\pr}]}\mu_{m_{E}(a,b)},\qquad \mu_{q}~\defequal ~\e\left[\left(W_{a,b}^{(i)}\right)^{q}\right],\label{E:wt-def}
\end{equation}
where the expectation is with respect to $\mu.$ In the proof of Proposition \ref{prop:2k_tO_reduction} we will consider sequences of compatible edge sets in the sense of the following definition. 
\begin{defn}
\label{def:edge_sequence}Let $d\in\bN$ and let $n_{0},n_{1},\ld,n_{d}\in\bN$. Let $\Si^{\ell}(n_{0},\ld,n_{d})$ denote the set of {\bf edge sequences} $E(1),\ld E(d)$ which satisfy:
\[E(i)\in\Si^{\ell}(n_{i-1},n_{i}),\qquad R(E(i-1))=L(E(i)),\,i=2,\ldots, d\]
The second condition ensures the endpoints of the edges of one layer are compatible with the edges from the next layer. Further, define for each $\ell$ the set of ordered paths:
\begin{equation}
\Ga^\ell  ~\defequal~[n_{0}]^\ell\times[n_{1}]^\ell\times\ld\times[n_{d}]^\ell.\label{E:Gamma-def}
\end{equation}
\end{defn}
\noindent Given $\ga\in \Ga^\ell$ define the edge sequence $E^\ga\in \Si^\ell(n_0,\ld,n_d)$ corresponding to $\ga$ by specifying the multiplicities
\begin{equation}
m_{E^\ga(i)}~\defequal~ m_{\ga(i-1),\ga(i)}. \label{E:path-to-edge-def}
\end{equation}
The formula \eqref{E:lr-paths} below will be used in the proof of Proposition \ref{prop:2k_tO_reduction}.
\begin{lem}
\label{lem:edge_sequence_path_counting}Let $d\in\bN$ and let $n_{0},n_{1},\ld,n_{d}\in\bN$
and $\ell\in\bN$. Consider $v\in [n_d]^\ell$ and any edge sequence $E\in\Si^{\ell}(n_{0},\ld,n_{d})$ with
\[R(E(d))=v.\]
Then, the number of ordered paths $\ga \in \Ga^{\ell}$ which have the same edge sequence as $E$ and have $\ga(d) = v$ is given by:
\begin{equation}
\abs{\left\{ \ga\in\Ga^\ell:\ E^\ga = E,\, \ga(d)=v\right\}}=\prod_{i=1}^{d}c_{\ell}\left(m_{E(i)}\right)\label{E:lr-paths}
\end{equation}
\end{lem}
\begin{proof}
The proof is by induction on $d.$ When $d=1$, the left hand side of \eqref{E:lr-paths} is precisely the number of $\ga(0)$ so that $\ga=(\ga(0),v)$ has $E^\ga = E,$ which by definition of $c_\ell$, equals $c_\ell(m_{E(1)})$. Let us now suppose we have proved the statement for $d=1,\ldots, D-1$ with $D\geq 2.$ By denoting $\ga(D-1)=\chi$, and counting the number of possibilities for $\ga(D)$ with $\ga(D-1)=\chi$, we write the left hand side of \eqref{E:lr-paths} as the sum
\begin{equation}
\sum_{\substack{\ch\in [n_{D-1}]^\ell\\ m_{\xi,v}=m_{E(D)}}} \abs{\left\{ \ga\in\Ga^\ell(n_0,\ld,n_{d-1})~|~ E^\ga = (E(1),\ld,E(D-1)),\, \ga(D-1)=\ch\right\} }. \label{E:induct_sum}
\end{equation}
Note that since $m_{\xi,v}=m_{E(D)},$ we find that $\ch$ coincides with the right endpoints  $R(E(D-1))$. Hence, by the inductive hypothesis, every term appearing in the sum from equation \eqref{E:induct_sum} is equal to $\prod_{i=1}^{D-1}c_{\ell}\left(m_{E(i)}\right)$ and does not depend on $\chi$ (since $c_{\ell}(m_{E(D-1)})$ depends only on the right endpoints of $E(D-1)$ and not on their order). The number of terms in the sum from equation \eqref{E:induct_sum} is exactly $c_\ell(m_{E(D)}),$ by the definition of $c_\ell$. The total is therefore $c_\ell(m_{E(D)}) \times \prod_{i=1}^{D-1}c_{\ell}\left(m_{E(i)}\right)$, completing the induction.
\end{proof}

\subsection{Proof of Proposition \ref{prop:2k_tO_reduction}}\label{S:proof-completion}
The first step in proving Proposition \ref{prop:2k_tO_reduction} is to express $\e\left[\norm{\mat{M}\vec{u}}^{2k}\right]$ as a sum over certain collections of $2k$ paths. 

\begin{defn}
Let $Q^{2k}$ be the set of $2k$ tuples of paths:
\[
Q^{2k}~\defequal~\left\{ \ga=\left(\ga_{1},\ga_{2},\ld\ga_{2k-1},\ga_{2k}\right)\ :\ \forall\,j , \,\, \ga_{2j-1}(d)=\ga_{2j}(d) \right\} \subset\Ga^{2k},
\]
where $\Ga^{2k}$ was defined in \eqref{E:Gamma-def}. Our notation is that if $\ga\in Q^{2k}$, then $\ga(i)\in[n_{i}]^{2k}$ is a $2k$-tuple for each $0\leq i\leq d$.
\end{defn}
\begin{lem}
\label{lem:expectation_to_2k_tuples}For a $2k$-tuple $x=\left(x_{1},\ld,x_{2k}\right)\in[n]^{2k}$,
let $\#x$ be the number of unique elements in $x$. Let $\vec{u}=(u_1,\ld,u_{n_0})\in\bR^{n_{0}}.$ Then: 
\begin{equation}
\e\left[\norm{\mat{M}\vec{u}}^{2k}\right]=\left(\prod_{i=0}^{d-1}\frac{1}{n_{i}^{k}}\right)\sum_{\ga\in Q^{2k}}u_{\ga(0)}\prod_{i=1}^{d}wt\left(m_{\ga(i-1),\ga(i)}\right)p^{\#\ga(i)-k},\qquad u_{\ga(0)} \defequal \prod_{i=1}^{2k} u_{\ga_i(0)}.\label{E:2k-paths}
\end{equation}
\end{lem}

\begin{proof}
Note that the entries of the $n_{d}\times n_{0}$ matrix $\mat{M}$ can
be written as a sum over certain paths in $\Ga$, namely:
\[
M_{a,b}=\sum_{\left\{ \ga\in\Ga:\ga(0)=b,\ga(d)=a\right\} }\prod_{i=1}^{d}\frac{1}{\sqrt{n_{i-1}p}}\xi_{\ga(i)}^{(i)}W_{\ga(i-1),\ga(i)}^{(i)},
\]
Using this interpretation in terms of paths, we obtain by indexing the starting points as $b_1,b_2 \in [n_0]$ and the ending point as $a \in [n_d]$, that we can write $\norm{\mat{M}\vec{u}}^{2}$ as a sum over $\ga \in Q^2$:
\begin{align*}
\norm{\mat{M}\vec{u}}^{2} & =\sum_{a\in[n_{d}]}\left(\sum_{b \in [n_0]} M_{a,b}u_b\right)^{2} =\left(\prod_{i=0}^{d-1}\frac{1}{n_{i}}\right)\sum_{\ga\in Q^{2}} \prod_{j=1}^2 u_{\ga_j(0)} \prod_{i=1}^{d}\frac{1}{\sqrt{p}}\xi_{\ga_{j}(i)}^{(i)}W_{\ga_{j}(i-1),\ga_{j}(i)}^{(i)},
\end{align*}
Similarly, the $k$-th power is then given by:
\begin{align*}
\norm{\mat{M}\vec{u}}^{2k} & =\left(\prod_{i=0}^{d-1}\frac{1}{n^k_{i}}\right)\sum_{\ga\in Q^{2k}} u_{\ga(0)}\prod_{j=1}^{2k} \prod_{i=1}^{d}\frac{1}{\sqrt{p}}\xi_{\ga_{j}(i)}^{(i)}W_{\ga_{j}(i-1),\ga_{j}(i)}^{(i)}.
\end{align*}

The result of the lemma follows by taking expectation of both
sides, using the independence of the random variables $\xi_b^{(i)},\, W_{a,b}^{(i)}$'s, and relations
\[\e\left[\prod_{j=1}^{2k} \frac{1}{\sqrt{p}} \xi_{\ga_{j}(i)}^{(i)}\right]=p^{\#\ga(i)-k},\qquad \e\left[\prod_{j=1}^{2k}W_{\ga_{j}(i-1)\ga_{j}(i)}^{(i)}\right]=wt\left(m_{\ga(i-1),\ga(i)}\right).\]
\end{proof}

\begin{defn}
Since the law $\mu$ of the entries of the matrices $\mat{W}^{(i)}$ is assumed to symmetric around $0,$ the odd moments of $\mu$ are all zero, and it will be useful to consider only edge sets that are ``even'' in the following sense:
\[ m_{E}\in\bN^{n\times n^{\prime}} \text{is \textbf{even }if }\quad \forall a,b\in[n]\times[n^{\pr}]\,\,m_{E}(a,b)\in\bN\text{ is even}\]
as well as to define the related sets
\begin{align}
Q_{even}^{2k}&\defequal Q^{2k}\cap\bigcap_{i=1}^{d}\left\{ m_{\ga(i-1),\ga(i)}\text{ is even}\right\}. \label{E:Q-even-def} \\
\Si^{2k}_{even}(n_0,\ld,n_d) &\defequal \Si^{2k}_{even}(n_0,\ld,n_d)\cap\bigcap_{i=1}^{d}\left\{ m_{E(i)}\text{ is even}\right\}. 
\end{align}
\end{defn}

\begin{lem}
With the same notation as in Lemma \ref{lem:expectation_to_2k_tuples}, we have
\begin{equation}
\e\left[\norm{\mat{M}\vec{u}}^{2k}\right]=\left(\prod_{i=0}^{d-1}\frac{1}{n_{i}^{k}}\right)\sum_{\ga\in Q_{even}^{2k}} u_{\ga(0)} \prod_{i=1}^{d}wt\left(m_{\ga(i-1),\ga(i)}\right)p^{\#\ga(i)-k}\label{E:even-path-moments}
\end{equation}
\end{lem}
\begin{proof}
Because the variables $W_{\alpha,\beta}^{(i)}$ are symmetric around $0$, all their odd moments vanish. Thus, in the expression \eqref{E:2k-paths}, only collections of paths in which every edge is traversed an even number of times given a non-zero contribution. What remains are exactly paths from $Q_{even}^{2k}$ by the definition in equation \eqref{E:Q-even-def}.
\end{proof}

\begin{proof}[Proof of Proposition \ref{prop:2k_tO_reduction}] 

Recall the definition of the edge sequences $\Si^{\ell}(n_0,\ld,n_d)$ and the notation $E^{\ga} \in \Si^{\ell}(n_0,\ld,n_d)$ for paths $\ga \in \Ga^{\ell}$ from Definition \ref{def:edge_sequence} (In this proof, we will use this definition when $\ell = 2k$ for paths $\ga \in \Ga^{2k}$ and when $\ell = k$ for paths $V \in \Ga^{k}$). Fix any $v \in [n_d]^k$. Let $\chi(v) \defequal (v_1,v_1,\ld,v_k,v_k) \in [n_d]^{2k}$ be $v$ with the entries doubled. For any function of edge sequences, $f : \Si^{2k}(n_0,n_1,\ld,n_d) \to \bR$, (it will be more convenient to write $f(E)=f(m_E)$, thinking of $f$ as a function of the multiplicities of the edge set), consider the following identity for sums over $\ga \in Q^{2k}_{even}$ that end at $\ga(d) = \chi(v)$: 
% that for any function $A(\ga)=A(E^\ga)$ that depends only on the edge set $E^\ga$ of $\ga$ we have
\begin{align*}
\sum_{\ga\in Q_{even}^{2k}, \ga(d)=\chi(v)}f(m_{E^\ga}) &= \sum_{E\in \Si_{even}^{2k}, R(E(d)) = \chi(v)}\abs{\{\ga\in Q_{even}^{2k}~|~E^\ga = E, \ga(d) = \chi(v) \}}f(m_E) \\
&=\sum_{V\in \Ga^k, V(d)=v}\frac{|\{\ga\in Q_{even}^{2k}~|~E^\ga = \chi(E^V), \ga(d) = \chi(v) \}|}{|V' \in \Ga^k~|~E^{V'} = E^{V}, V'(d)=v\}|} ~f(2m_{E^V}).
\end{align*}
Here $\chi(E)$ denotes doubling all the edges (i.e. the multiplicities double $m_{\chi(E)} = 2m_{E}$) and  we have used the fact that every even edge sequence $E \in \Si^{2k}_{even}(n_0,\ld,n_d)$ arises by taking a sequence $V \in \Ga^k$ and doubling the multiplicity of the edges. (Note that there may be multiple choices of $V$ for each $E \in \Si^{2k}_{even}(n_0,\ld,n_d)$, which is why we have to divide be the size of this set to account for this many-to-one-ness.) We now apply Lemma \ref{lem:edge_sequence_path_counting} to both the numerator (with $\ell = 2k$) and the denominator (with $\ell =k$) to see that the enumeration depends only on the edge set $E^V$ and the endpoints of the last layer $V(d)=v$:
\[\frac{|\{\ga\in Q_{even}^{2k}~|~E^\ga = \chi(E^V), \ga(d) = \chi(\xi) \}|}{|\{V' \in \Ga^k~|~E^{V'} =  E^{V}, V'(d) = v \}|}=\prod_{i=1}^d \frac{c_{2k}(2m_{V(i-1),V(i)})}{c_{k}(m_{V(i-1),V(i)}) }.\]
Summing over all possible endpoints $v \in [n_d]^k$ now gives the identity:
\begin{equation}
\sum_{\ga\in Q_{even}^{2k}} f(m_{E^\ga})= \sum_{V\in \Ga^k} \prod_{i=1}^d \frac{c_{2k}(2m_{V(i-1),V(i)})}{c_{k}(m_{V(i-1),V(i)}) } f(2m_{E^V}).
\end{equation}
Finally, using this identity on \eqref{E:even-path-moments}, with $f$ being the function that appears inside the sum over $Q^{2k}_{even}$, gives the desired result of Proposition \ref{prop:2k_tO_reduction}. \end{proof}

\subsection{Completion of Proof of Equation \eqref{E:Mu-moments}}\label{S:proof-completion}
\begin{defn}
We think of the sum in Proposition \ref{prop:2k_tO_reduction} as an expectation over discrete random variables $V(i)$. Specifically, we write:
\begin{equation}
\e\left[\frac{n_0^k}{n_{d}^{k}}\norm{\mat{M}\vec{u}}^{2k}\right]  =\left(\prod_{i=1}^{d}\frac{1}{n_{i}^{k}}\right)\sum_{V\in[n_{0}]^{k}\times\ld\times[n_{d}]^{k}}u_{V(0)}^2 \prod_{i=1}^{d}C(V(i-1),V(i)) \defequal \cE_{\vec{u}}\left[\prod_{i=1}^d C(V(i-1),V(i))\right],\label{E:Eu}
\end{equation}
where $\cE_{\vec{u}}$ is defined to be the expectation with respect to a product measure on sequences $V$, in which the $k$ entries of $V(0)\in [n_0]^k$ are chosen i.i.d. from the measure $(u_1^2,\ld,u_{n_0}^2)$, (i.e. $\cP(V_a(0) = j) = u^2_j$ for every $1\leq j \leq n_0$; this is a probability measure since $u$ is a unit vector), and the $k$ entries of $V(i)\in [n_i]^k$ are chosen i.i.d. from the uniform measure on $[n_i]$ for every $i\geq 1$. (i.e. $\cP( V(i) = v ) = \frac{1}{n_i^k}$ for any $v \in [n_i]^k$, $i \geq 1$). In order to prove that the rightmost product in \eqref{E:Eu} equal the right hand side of \eqref{E:Mu-moments}, we introduce some notation. Namely, for $n\in \mathbb N,$ we partition the set $\left[n\right]^{k}$ into three pieces: 
\[
[n]^{k}\defequal U_{n}\dot{\cup}P_{n}\dot{\cup}B_{n}.
\]
Informally, $U$ stands for ``unique entries'', and consists of those $k$-tuples with no repeated entries; $P$ stands for ``one pair'' and consists of those $k$-tuples with exactly one repeated entry; $B$ stands for ``bad'' and consists of everything else. Formally, 
\begin{align*}
U_{n} & ~\defequal~\left\{ v:\!v_{i}\neq v_{j}\text{ for }i\neq j\right\} \subset[n]^{k},\quad P_{n}  ~\defequal~\left\{ v:\!\exists!i<j\text{ s.t. }v_{i}=v_{j}\right\} \subset[n]^{k},\quad B_{n}  ~\defequal~[n]^{k}\backslash\left(U_{n}\cup P_{n}\right).
\end{align*}
\end{defn}

%The first step in studying the expectation \eqref{E:Eu} is to explicitly compute the effect of the non-uniform measure $\{u_1^2,\ld,u_{n_0}^2\}$ that determines the distribution of $V(0).$ This is done using the following two Lemmas.

\begin{lem}\label{L:probs}
For each $i\geq 1$, under the uniform measure on $[n_i]^k$, each random variable $V(i)$ has the following probabilities for the events $\left\{ V(i) \in U_{n_i}\right\}$,$\left\{V(i) \in P_{n_i}\right\}$,$\left\{V(i) \in B_{n_i}\right\}$:

%\begin{align*}
%\cP\left(V(i) \in U_{n_i}\right) &= 1- \binom{k}{2}\frac{1}{n_i}+O(n_i^{-2})\\
%\cP\left(V(i) \in P_{n_i}\right) &= \binom{k}{2}\frac{1}{n_i}+O(n_i^{-2})\\
%\cP\left(V(i) \in B_{n_i}\right) &= O(n_i^{-2})
%\end{align*}

\[\cP(V(i) \in *_{n_i})=
\begin{cases}
  1- \binom{k}{2}\frac{1}{n_i}+O(n_i^{-2}),& \text{ if } *=U\\
  \binom{k}{2}\frac{1}{n_i}+O(n_i^{-2}),&\text{ if } *=P\\
O(n_i^{-2}),&\text{ if } *=B
\end{cases}.
\]
\end{lem}
\begin{proof}
The proof is an elementary exercise in discrete probability.
\end{proof}

\begin{lem}\label{lem:C_calculations} Subdivide the ``one pair'' set by which indices are paired: $P_n \defequal \cup_{a\neq b}P_n(a,b)$ where $P_n(a,b) = P_n \cap \left\{ x \in [n]^k : x_a = x_b  \right\}$ for $a \neq b$. Then for each $k\geq 1,\, i=1,\ld,d,$ we have
\[C\left(V(i-1),V(i)\right)=
\begin{cases}
  1,&\quad V(i)\in U_{n_i}\\
3/p,&\quad V(i)\in P_{n_i}(a,b) \text{ if }V_a(i-1) \neq V_b(i-1)\\
\mu_4/p,&\quad V(i)\in P_{n_i}(a,b) \text{ if }V_a(i-1) = V_b(i-1)\\
\Theta(1),&\quad \text{otherwise}
\end{cases},
\]
%\[C\left(V(i-1),V(i)\right)=
%\begin{cases}
%  1,&\quad V(i)\in U_{n_i}\\
%3/p,&\quad V(i)\in P_{n_i}\text{ with }V_\ell(i) = V_m(i),\,m\neq \ell, \text{ and }V_\ell(i-1) \neq V_m(i-1)\\
%\mu_4/p,&\quad V(i)\in P_{n_i}\text{ with }V_\ell(i) = V_m(i),\,m\neq \ell, \text{ and }V_\ell(i-1) = V_m(i-1)\\
%\Theta(1),&\quad \text{otherwise}
%\end{cases},
%\]

where the implicit constant in $\Theta(1)$ is bounded below by $1$ and above by $\mu_{2k}(2k-1)!!p^{1-k}.$
\end{lem}

\begin{proof}
This is an elementary calculation from the definition of $C$. If $V(i) \in U_{n_i}$, $\#V(i) = k$ and the multiplicities of edges in the edge set $E^{V(i-1),V(i)}$ are all $1$ which makes the combinatorial factor in $C(V(i-1),V(i))$ equal to $1$, and every edge is covered exactly twice giving a factor of $\mu^k_2 = 1$ in the weight term. If $V(i) \in P_{n_i}$, $\#V(i) = k-1$ giving a factor of $\frac{1}{p}$. Moreover, in this case, when the indices which are paired in $V(i)$ are also paired in $V(i-1)$, all the combinatorial factors are again $1$, and the weight term is $\mu^{k-1}_2 = \mu_4$.  If the paired indices from $V(i)$ are not paired in $V(i-1)$, then there the combinatorial term is $1^{k-2}\frac{\binom{2\cdot2}{2\cdot1}}{\binom{2}{1}} = 3$, and the weight term is $\mu^k_2 = 1$. 
\end{proof}

\begin{lem}\label{L:u-reduce}
We have
\begin{equation}
  \label{eq:2}
  \cE_{\vec{u}}\left[\prod_{i=1}^d C(V(i-1),V(i))\right] = \psi_U T_U + \psi_P T_P + \psi_B T_B, 
\end{equation}
where the quantities $\psi_U, T_U,\psi_P, T_P,\psi_B, T_B$ are:
\[ \psi_*~\defequal~
\begin{cases}
  1 - \binom{k}{2}\frac{1}{n_1}+O(n_1^{-2}),&\text{ if }*=U\\
  \binom{k}{2}\frac{1}{n_1}\left(\frac{\mu_4-3}{p}\norm{\vec{u}}_4^4 +\frac{3}{p}\right)\left(1+O(n_1^{-1})\right),&\text{ if }*=P\\
  \Theta(n_1^{-2}),&\text{ if }*=B
\end{cases}
,\qquad
T_*~\defequal~ \cE\left[\prod_{i=2}^d C(V(i-1),V(i))~\big|~ V(1)\in *_{n_1}\right]
\]
\end{lem}
\begin{proof}
Note that for any fixed $j=1,\ldots, d-1$ and $v\in [n_j]^k$, we have the following conditional independence of layers before $V(j)$ and after $V(j)$:
\[  \cE_{\vec{u}}\left[\prod_{i=1}^d C(V(i-1),V(i))~\big|~ V(j)=v\right]=  \cE_{\vec{u}}\left[\prod_{i=1}^{j} C(V(i-1),V(i))~\big|~ V(j)=v\right]\cE\left[\prod_{i=j+1}^{d} C(V(i-1),V(i))~\big|~ V(j)=v\right],\]
where in the second term we write $\cE$ instead of $\cE_{\vec{u}}$ since the measure no longer depends on $u.$ Applying this with $j=1,$ we find
\begin{equation}\label{E:u-dep}
  \cE_{\vec{u}}\left[\prod_{i=1}^d C(V(i-1),V(i))\right] = \sum_{v\in [n_1]^k}\cE_{\vec{u}}\left[C(V(0),v)\right] \cE\left[\prod_{i=2}^d C(V(i-1),V(i)) \given{V(1) = v}\right] \cP\left(V(1)=v\right),
\end{equation}
where in the second term the random variables $V(2),\ld,V(d)$ are uniform and do not depend on $u$ or $v$. An elementary probability computation using Lemma \ref{lem:C_calculations} and the measure $\{u^2_0,\ldots,u^2_{n_0}\}$ on $V(0)$ shows that
\begin{equation}
 \cE_{\vec{u}}\left[C(V(0),v)\right]=\one\{v\in U_{n_1}\} + \one\{v\in P_{n_1}\}\left(\frac{\mu_4-3}{p}\norm{\vec{u}}_4^4 + \frac{3}{p}\right)+\one\{v\in B_{n_1}\}\Theta(1),\label{E:u4-term}
\end{equation}
where the implicit constant in the last term is bounded below by $1$ and above by $\mu_{2k}(2k-1)!!p^{1-k}.$ Combining the result of Lemma \ref{L:probs}, with \eqref{E:u-dep} and \eqref{E:u4-term}, proves Lemma \ref{L:u-reduce}.

%Moreover, note that if $V\sim \mathrm{Unif}\left([n]^{k}\right)$, then
%\begin{align*}
%\cP\left(V\in U_{n}\right) & =1-\binom{k}{2}\frac{1}{n}+O\left(n^{-2}\right),\quad \cP\left(V\in P_{n}\right) =\binom{k}{2}\frac{1}{n}+O\left(n^{-2}\right),\quad \cP\left(V\in B_{n}\right) =O\left(n^{-2}\right).
%\end{align*}
\end{proof}

\begin{lem}\label{L:UB-LB}
%Define the events
%\[ * (i) ~\defequal~ \left\{ V(i)\in *_{n_i} \right\},\qquad *\in \{U,P,B\}\]
%which are disjoint for each $i\geq 1$ and, 

Recall the definition of $T_*~\defequal~ \cE\left[\prod_{i=2}^d C(V(i-1),V(i))~\big|~ V(1)\in *_{n_1}\right]$ for $\ast \in \{U,P,B\}$ from Lemma \ref{L:u-reduce}. Define the indicator functions
\[X_i~\defequal~ \one\{V(i) \in P_{n_i}\},\qquad Y_i~\defequal~\one\{V(i) \in B_{n_i}\}.\]
Then, for any choice of the label $\ast \in \{U,B,P\}$, we have that:
\begin{equation}
\mathbb \cE\left[\prod_{i=2}^d \alpha^{X_i}\gamma_{\wedge}^{X_{i-1}X_i}~\big|~ V(1) \in \ast_{n_1} \right]~\leq~ T_*~\leq~ \widehat{K}^{\one\{*=B\}}\mathbb \cE\left[\alpha_*^{X_2}\widehat{K}^{Y_2}\prod_{i=3}^d \alpha^{X_i}\gamma_{\vee}^{X_{i-1}X_i}\widehat{K}^{Y_i}~\big|~ V(1) \in \ast_{n_1}\right],\label{E:T-star-bounds}
\end{equation}
where we've introduced
\[\alpha ~\defequal~ \frac{3}{p},\qquad\gamma_{\vee} ~\defequal~ 1\vee \frac{\mu_4}{3},\qquad\ga_\wedge~\defequal~ 1\wedge \frac{\mu_4}{3},\qquad \widehat{K}~\defequal~ K\left( \frac{3\vee \mu_4}{p}\right),\qquad \alpha_* \defequal
\begin{cases}
  \alpha,&\quad *=U,B\\
  \gamma_{\wedge},&\quad *=P
\end{cases}.\]
\end{lem}
\begin{proof}[Proof of Lemma \ref{L:UB-LB}]
By using the possible values for $C$ computed in Lemma \ref{lem:C_calculations} and the definition of $T_{\ast}$, we have that for any label $\ast \in \{U,P,B\}$:
\[\cE\left[\prod_{i=2}^d \left( \frac{3}{p}\right)^{U\ra P(i)}\left(\frac{3\wedge \mu_4}{p}\right)^{\widehat{U}\ra P(i)}~\big|~ V(1)\in \ast_{n_1} \right]\leq T_*\leq \cE\left[\prod_{i=2}^d \left( \frac{3}{p}\right)^{U\ra P(i)}\left(\frac{3\vee \mu_4}{p}\right)^{\widehat{U}\ra P(i)}K^{B(i)}~\big|~ V(1)\in \ast_{n_1} \right],\]
where we've abbreviated
\[U\ra P(i) \defequal \one\left\{ V(i-1)\in U_{n_{i-1}}, V(i)\in P_{n_i} \right\},\quad \widehat{U}\ra P(i) \defequal \one\left\{ V(i-1)\not \in U_{n_{i=1}}, V(i)\in P_{n_i} \right\},\quad B(i) \defequal \one\left\{ V(i)\in B_{n_i} \right\}.\]
Note that 
\[\left(\frac{3}{p}\right)^{U\ra P(i)}\left(\frac{3\vee \mu_4}{p}\right)^{\widehat{U}\ra P(i)} \leq\left(\frac{3}{p}\right)^{P(i)}\left(1\vee \frac{\mu_4}{3}\right)^{P\ra P(i)}\left(\frac{3\vee \mu_4}{p}\right)^{B(i-1)}\leq \al^{X_i}\ga_\vee^{X_{i-1}X_i}\widehat{K}^{Y_{i-1}}\]
This proves the upper bound in \eqref{E:T-star-bounds}. The lower bound similarly follows:
\[ \left( \frac{3}{p}\right)^{U\ra P(i)}\left(\frac{3\wedge \mu_4}{p}\right)^{\widehat{U}\ra P(i)}\geq  \alpha^{X_i}\ga_\wedge^{X_{i-1}X_i}.\]
\end{proof}

\begin{proof}[Completion of Proof of Relation \eqref{E:Mu-moments}] We first notice, by application of the elementary probability estimate recorded in Lemma \ref{lem:classic_probability}, that the upper and lower bounds on $T_\ast$ given in Lemma \ref{L:UB-LB} are equal up to small errors. We have for $\ast \in \{ U,P,B \}$:
\begin{equation}
\Pi~\leq~ T_* ~\leq~ \widehat{K}^{\one\{*=B\}}(1+O(n_2)^{-1})^{\one\{*=P\}}\Pi,\qquad \Pi ~\defequal~
\prod_{i=2}^d\left\{ 1+\left(\frac{3}{p}-1\right)\binom{k}{2}\frac{1}{n_i} + O\left(\frac{1}{n_i^2}\right)\right\} \label{E:u-indep}
\end{equation}
(where $\widehat{K}$ is as in Lemma \ref{L:UB-LB}). Finally, putting these values for $T_U,T_P,T_B$ into the result of Lemma \ref{L:probs} we see:

\begin{align*}
&\cE_{\vec{u}}\left[\prod_{i=1}^d C(V(i-1),V(i))\right] \\
 =&  \left\{ 1 + \binom{k}{2}\frac{1}{n_1}\lr{\frac{3}{p}-1 + \frac{\mu_4-3}{p}\norm{\vec{u}}_4^4 } + O\lr{n^{-2}_1} + O\lr{n^{-2}_2} \right\} \prod_{i=2}^d\left\{ 1+\left(\frac{3}{p}-1\right)\binom{k}{2}\frac{1}{n_i} + O\lr{n^{-2}_i}\right\} \\
=&\exp\left( \binom{k}{2}\be + O\lr{\sum_{i=1}^d\frac{1}{n^2_i}} \right)
\end{align*}
The last line follows from the elementary fact for exponentials $e^{x - \half x^2} \leq 1+x \leq e^{x}$ for $x \geq 0$.
%
%\[\prod_{i=1}^d(1+x_i + \ep_i)=\exp\left[\sum_{i=1}^d x_i + \Theta(\sum_{i=1}^d (\ep_i+x_i^2))\right],\]
%
%which hold when $x_i+\ep_i>\delta \geq 0$ for all $i$ and the implicit constant depends only on $\delta.$
 \end{proof} 

\subsection{An elementary probability estimate}

\begin{lem}
\label{lem:classic_probability}Let $A_{0},A_{1},\ld,A_{d}$ be independent
events with probabilities $p_{0},\ld,p_{d}$ and $B_0,\ld, B_D$ be independent events with probabilities $q_0,\ld,q_d$ such that
\[A_j\cap B_j=\emptyset,\qquad \forall j=0,\ld,d.\]
Denote by $X_{i}$ the indicator that the event $A_{i}$ happens, $X_{i}:=\one\left\{ A_{i}\right\} $, and by $Y_i$ the indicator that $B_i$ happens, $Y_i=\one\{B_i\}$. Further, fix for every $i\in 1,\ld,d$ some $\al_i \geq 1,K_i\geq 1$ as well as $\ga_i>0$. Define
%\[Z_d:=\al_*^{X_0}K^{Y_0}\al^{\sum_{i=1}^{d}X_{i}}\ga^{\sum_{i=1}^{d}X_{i-1}X_{i}}K^{\sum_{i=1}^dY_i}\]
\[Z_d~\defequal~ \prod_{i=1}^d \alpha_i^{X_i}\ga_i^{X_{i-1}X_i}K_i^{Y_i}.\]
Then, if $\ga_i\geq1$ for every $i$, we have:
\begin{align}
\e\left[Z_d\right]~\leq~\prod_{i=1}^d\left(1+p_i(\al_i-1)+q_i(K_i-1)+p_ip_{i-1}\al_i\al_{i-1}\ga_{i-1}(\ga_{i}-1)\right)\label{eq:ineq_1},
\end{align}
where by convention $\alpha_0=\ga_0=1.$ In contrast, if $\ga_i\leq 1$ for every $i$, we have:
\begin{equation}
\e[Z_d]~\geq~  \prod_{i=1}^d\left(1+p_i(\al_i-1)+p_ip_{i-1}\al_{i-1} \al_i(\ga_i-1)\right)\label{eq:ineq_2}
\end{equation}
\end{lem}

\begin{proof}[Proof of Lemma \ref{lem:classic_probability}]
The proof goes by induction on $d$. The base case $d=1$ can be computed directly
\[\e[Z_1]=1+ p_1(\alpha_1-1) + q_1(K_1-1) +p_0p_1\al_1(\ga_1-1),\]
which is verified to obey the stated inequalities under the convention $\al_0 = \ga_0 = 1$. To see the induction step, suppose that $d \geq 2$. Define the filtration $\cF_{j}=\si\left(X_{0},Y_0,\ld,X_{j},Y_j\right)$. We have, from the definition of $Z_i$ that
\[Z_d = Z_{d-1} \alpha_d^{X_d}\ga_d^{X_dX_{d-1}}K_d^{Y_d}.\]
We compute by directly examining what happens when $X_{d-1} = 0$ and when $X_{d-1} = 1$, that
\begin{equation}
\e[Z_d]=\e\left[\e\left[Z_d~|~\cF_{d-1}\right]\right]=\e\left[Z_{d-1}\right]\left(1+ p_d(\alpha_d-1) + q_d(K_d-1) \right)+\e\left[X_{d-1}Z_{d-1}\right]\al_d(\ga_d-1)p_d.\label{E:z-one-step}
\end{equation}
Now notice that, since $X_{d-1}Z_{d-1}$ vanishes when $X_{d-1}=0$, and since $A_{d-1}\cap B_{d-1}=\emptyset$, we have that
\begin{equation}
 X_{d-1} Z_{d-1} = X_{d-1} \al_{d-1}^{X_{d-1}}\ga_{d-1}^{X_{d-1}X_{d-2}}K_{d-1}^{Y_{d-1}} Z_{d-2} = X_{d-1} \alpha_{d-1} \ga_{d-1}^{X_{d-2}} Z_{d-2} \label{E:XZ}
\end{equation}
Hence, when $\ga_d \geq 1$, since $Z_j\leq Z_{j+1}$ for every $j$, we have the estimate
\begin{align*}
\e[X_{d-1}Z_{d-1}]&= \al_{d-1}\ga_{d-1} p_{d-1} \e\left[Z_{d-2}\right]\leq \al_{d-1}\ga_{d-1} p_{d-1} \e\left[Z_{d-1}\right]
\end{align*}
and hence obtain from equation \eqref{E:z-one-step}
\[\e[Z_d]\leq \e\left[Z_{d-1}\right]\left(1+ p_d(\alpha_d-1) + q_d(K_d-1) +p_dp_{d-1}\al_d\al_{d-1}(\ga_d-1)\ga_{d-1} \right),\qquad d\geq 2.\]
which is the desired inequality to prove the induction step for the upper bound. 
%Iterating inequality yield if $\ga_j\geq 1$ for all $j$ that
%\[\e\left[Z_d\right]\leq \prod_{i=1}^d\left(1+p_i(\al_i-1)+q_i(K_i-1)+p_ip_{i-1}\al_i\al_{i-1}\ga_{i-1}(\ga_{i}-1)\right),\]
%where we set $\al_0=\ga_0=1.$
% The lower bound \eqref{eq:ineq_2} is similar. 
To see the lower bound, we will actually prove the lower bound for the sequence
\[\widehat{Z}_d~\defequal~ \prod_{i=1}^d \al_i^{X_i}\ga_i^{X_{i-1}X_i}=\widehat{Z}_{d-1} \al_d^{X_d}\ga_d^{X_{d-1}X_d},\]
This is what one gets if all the parameters $K_d$ are equal to $1$, so clearly $Z_i \geq \widehat{Z}_i$ and it is sufficient to bound this new sequence. Notice that since $\ga < 1$, we have $\widehat{Z}_{d-2}~\leq~ \ga_{d-1}^{-1}\widehat{Z}_{d-1}$, so applying equation \eqref{E:XZ} to this sequence gives that:
\begin{align*}
\e[X_{d-1}\widehat{Z}_{d-1}]&= \al_{d-1}\ga_{d-1} p_{d-1} \e\left[\widehat{Z}_{d-2}\right]\leq \al_{d-1} p_{d-1} \e\left[\widehat{Z}_{d-1}\right] 
\end{align*}
so by equation \eqref{E:z-one-step} applied to the $\widehat{Z}_i$ sequence, we have then (keeping in mind $\ga_{d} - 1 < 0$ reverses the inequality):
\[\e[\widehat{Z}_d]\geq \e\left[\widehat{Z}_{d-1}\right]\left(1+ p_d(\alpha_d-1) +p_dp_{d-1}\al_d\al_{d-1}(\ga_d-1) \right),\qquad d\geq 2.\]
which is the desired inequality for the induction step on the lower bound.
%
%and also that
%\[\e\left[\widehat{Z}_d~|~\cF_{d-1}\right]=\widehat{Z}_{d-1}\left(1+p_d(\al_d-1)+X_{d-1}p_d\al_d(\ga_d-1)\right).\]
%Note that since $\ga_i\leq 1$, we have 
%\[\ga_i^{X_{i-1}X_i}~\geq~ \ga_i^{X_i}\qquad \text{and}\qquad \widehat{Z}_{d-2}~\geq~ \alpha_{d-1}^{-1}\widehat{Z}_{d-1}.\]
%Hence, 
%\[\e\left[X_{d-1}\widehat{Z}_{d-1}\right]~=~\e\left[\e\left[X_{d-1}\widehat{Z}_{d-1}~|~\cF_{d-2}\right]\right]~\geq~ \e\left[\widehat{Z}_{d-1}\right]p_{d-1}\ga_{d-1}.\]
%Putting the preceding estimates together, we have
%\[\e\left[\widehat{Z}_d\right]\geq \e\left[\widehat{Z}_{d-1}\right]\left(1+p_d(\al_d-1)+p_dp_{d-1}\al_d(\ga_d-1)\ga_{d-1}\right).\]
%Iterating this inequality yields
%\[\e\left[\widehat{Z}_d\right]\geq\prod_{i=1}^d\left(1+p_i(\al_i-1)+p_ip_{i-1}\al_i(\ga_i-1)\ga_{i-1}\right),\]
%where we've again set $\ga_0=\al_0=1.$ This is precisely the lower bound \eqref{eq:ineq_2}, completing the proof of Lemma \ref{lem:classic_probability}.
\end{proof}

\section{Proof of Theorem \ref{T:main}: Quantitative Martingale CLT} \label{S:CLT}
%\subsection{Overview of the Proof Strategy}
In the section, we explain the proof of the distribution estimates in equation \eqref{E:KS-dist} in Theorem \ref{T:main} modulo the proof of several key technical results, which are proved in Sections \ref{S:alpha-normality-pf} and \ref{S:KS-pf} below. 

We first recall the notation. Namely, fix $0 <p\leq1$ and consider a fixed measure $\mu$ satisfying \eqref{E:mu-def}. For every $i=1,\ldots, n$ take independent random $n_{i}\x n_{i-1}$ matrices $\mat{W}^{(i)}$ with all the entries of $\mat{W}^{(i)}$ drawn i.i.d. from $\mu$ and for each $i=1,\ldots, d$, consider $n_i \times n_i$ diagonal matrices $\mat{D}^{(i)} = \text{diag}\lr{ \xi^{(i)}_1,\ldots,\xi^{(i)}_{n_i} }$ where $\xi_a^{(i)}$ are iid $\left\{ 0,1\right\}-$valued independent Bernoulli$(p)$ variables $\p\lr{\xi = 1}=1-\p\lr{\xi=0} = p$. The key objects of study are, for $i=0,1,\ld$, the random $n_i\x n_0$ matrices
\[\mat{M}^{(i)}~\defequal~ \mat{D}^{(i)}\lr{p n_{i-1}}^{-1/2}\mat{W}^{(i)}\cdots \mat{D}^{(1)}\lr{p n_0}^{-1/2}\mat{W}^{(1)},\qquad \mat{M}^{(0)}~ \defequal~ \mathrm{\mat{Id}}_{n_0}.\]
The estimates \eqref{E:KS-dist} concern the distribution, for any fixed unit vector $\vec{u}^{(0)} \in \bR^{n_0}$, of
\[ \vec{u}^{(i)} \defequal \sqrt{\frac{n_{i-1}}{n_{i}}} \mat{M}^{(i)} \vec{u}^{(0)} \in
\bR^{n_i}\]
Notice that the sequence $\vec{u}^{(i)}$ is equivalently defined recursively as:
\begin{equation}
\vec{u}^{(i)} = \lr{pn_{i}}^{-1/2}\mat{D}^{(i)}\mat{W}^{(i)}\vec{u}^{(i-1)},\qquad \vec{u}^{(0)}=\vec{u}. \label{eq:u_rec}
\end{equation}
%= \mat{D}^{(i)}\lr{p n_{i}}^{-1/2}\mat{W}^{(i)}\cdots D^{(1)}\lr{p
%n_1}^{-1/2}W^{(1)} \vec{u}^{(0)}
With this notation the relation \eqref{E:KS-dist} we seek to show becomes the statement that for every $m\geq 1$ 
\begin{align}
\notag &d_{KS}\lr{\ln\lr{\big|\big|\vec{u}^{(d)}\big|\big|_2^2},~ \cN(-\half \beta,\beta)}  \\ &\qquad= O\lr{\beta^{-1}\sum_{j=1}^{d}n_j^{-2}+\lr{\beta^{-2}\sum_{j=1}^{d}n_j^{-2}}^{1/5}+\lr{\beta^{-1/2}\sum_{j=1}^{d}n_j^{-2}}^{1/2}+\sum_{j=1}^d n_i^{-m}+\sum_{i=1}^d p^{n_i}},\label{E:full-KS-goal}
\end{align}
where $\cN(\mu,\sigma^2)$ is the Gaussian, and
\[ \be \defequal  \left(\frac{3}{p} - 1 \right)\sum_{i=1}^{d} \frac{1}{n_i} + \frac{\mu_4 - 3}{p n_1} \big|\big| \vec{u}^{(0)} \big|\big|^4_4. \] 
The idea of the proof is to look at the quantity $\ln\lr{\norm{\vec{u}^{(d)}}_2^2}$ as the value of a martingale at time $d$ with respect to the filtration
\[\cF_{i} \defequal \si\left(\mat{W}^{(1)},\ld,\mat{W}^{(i)},\mat{D}^{(1)},\ld,\mat{D}^{(i)}\right),\]
i.e. $\cF_{i}$ the sigma algebra generated by the random variables in the first $i$ layers. The basic idea of our proof is to deduce the approximate normality of $\ln\lr{\norm{\vec{u}^{(d)}}_2^2}$ by applying a martingale CLT with rate (see Theorem \ref{T:MCLT}). Specifically, note that $\ln\lr{\norm{\vec{u}^{(0)}}_2^2}=0$, since $\vec{u}^{(0)}$ is a unit vector. Hence, $\ln\lr{\norm{\vec{u}^{(d)}}_2^2}$, is a telescoping sum (modulo the complication discussed below that $\norm{\vec{u}^{(i)}}$ could vanish):
\begin{equation}
\ln\lr{\big|\big|\vec{u}^{(d)}\big|\big|_2^2}= \sum_{i=1}^d \ln{\left(\frac{\norm{\vec{u}^{(i)}}_2^2}{\norm{\vec{u}^{(i-1)}}_2^2}\right)},\label{E:telescope}
\end{equation}
%we will study in detail the associated martingale difference sequence 
%\begin{equation}
%X^{(i)}\defequal \ln\lr{\norm{\vec{u}^{(i)}}_2^2}-\E{\ln\left(\norm{\vec{u}^{(i)}}_2^2\right)\given{\cF_{i-1}}}=\ln{\left(\frac{\norm{\vec{u}^{(i)}}_2^2}{\norm{\vec{u}^{(i-1)}}_2^2}\right)}-\E{\ln{\left(\frac{\norm{\vec{u}^{(i)}}_2^2}{\norm{\vec{u}^{(i-1)}}_2^2}\right)}\given{\cF_{i-1}}}.\label{E:X-def}
%\end{equation}
and we will think of each entry of the sum as an increment. By subtracting off the conditional means, this will yield a martingale difference sequence which can be analyzed. It will turn out that the variance of these increments satisfy:
\[\var\lr{ \ln \lr{ \frac{\norm{\vec{u}^{(i+1)}}_2^2}{\norm{\vec{u}^{(i)}}_2^2} } ~\bigg|~ \cF_{i} }\approx \lr{\frac{3}{p}-1}\frac{1}{n_{i+1}}+\frac{\mu_4 - 3}{p n_{i+1}} \frac{\norm{\vec{u}^{(i)}}_4^4}{\norm{\vec{u}^{(i)}}_2^4}+ O(n_{i}^{-2}).\]
For $i \geq 1$, we will typically have that 
\[\frac{\norm{\vec{u}^{(i)}}_4^4}{\norm{\vec{u}^{(i)}}_2^4} \approx \frac{1}{n_i},\] 
and therefore the term involving the fourth moment $\mu_4$ will be of size $O(n_i^{-2})$ for all except the first layer when $i = 0$. The sum of these increment variances is precisely our variance parameter $\beta$ (modulo terms like $n_i^{-2}$). This informally explains the appearance of $\norm{\vec{u}^{(0)}}_4^4$ in the formula for $\beta$, and why the terms from other layers do not depend on the higher moments of $\mu$. 

To give a precise proof of \eqref{E:full-KS-goal}, we must deal with a wrinkle in the strategy described above: with a small but positive probability the vectors $\vec{u}^{(i)} = 0$, making the ratio of the norms of the vectors $\vec{u}^{(i)},\vec{u}^{(i-1)}$ in \eqref{E:telescope} undefined. Since the weight matrices $\mat{W}$ are assumed to have no atoms, this can only happen if the Bernoulli variables are all equal to zero. To take this into account, we define the events
\begin{equation}
A_{i}\defequal\set{S^{(j)}\neq 0,\,\forall j\leq i},\qquad \p(A_{i})=\prod_{j=1}^{i}\lr{1- p^{n_j}}=1 + O\lr{\sum_{j=1}^i p^{n_j}},\label{E:A-def}
\end{equation}
where we've abbreviated
\begin{align}
\S^{(i)} & \defequal \big|\big|\vec{u}^{(i)}\big|\big|_2^2 \label{E:S-def} 
\end{align}
In addition, we will find it convenient to fix a truncation level $0<\alpha<1,$ and set
\[\ln_{\al}(t)\defequal
\begin{cases}
  \ln(t),&~~t\geq \alpha\\
\ln(\alpha),&~~t\in(0,\alpha)
\end{cases}
\] 
We will study the sequence of martingale increments
\[X_i \defequal \ln_\alpha\lr{\frac{S^{(i)}}{S^{(i-1)}}}\one_{A_{i-1}}-\E{\ln_\alpha\lr{\frac{S^{(i)}}{S^{(i-1)}}}\one_{A_{i-1}}\given{\cF_{i-1}}},\]
that coincide, with high probability, with the martingale difference sequence associated to $\ln(S^{(i)})$ (see Lemma \ref{L:prob-est}), where by convention we define the product $\ln_\alpha\lr{S^{(i)} /S^{(i-1)}}\one_{A_{i-1}}$ is zero on the event $A^c_{i-1}$ when $S^{(i-1)}=0.$ To prove the approximate normality of $\ln(||\vec{u}^{(i)}||_2^2),$ we first prove the approximate normality of $\sum_i X_i$ in the following Proposition.  

\begin{prop}\label{P:alpha-normality}
We have that:
\[ \sum_{i=1}^d \E{X_i^2} = \be + O\lr{\sum_{i=1}^d n^{-2}_i}\]
Moreover, for any fixed $0<\al<1$, the sum $\sum_{i=1}^{d} X_i $ is approximately normally distributed in the sense that
\begin{equation}\label{E:alpha-normality}
  d_{KS}\lr{ \frac{\sum_{i=1}^d X_i}{\sqrt{\sum_{i=1}^d \E{X_i^2}}}, \,\,\mathcal N(0,1)}=O\left( \beta^{-2}\sum_{i=1}^{d}\frac{1}{n_{i}^{2}} \right)^{1/5}. 
\end{equation}
%\begin{equation}\label{E:alpha-normality}
%  d_{KS}\lr{\frac{\sum_{i=1}^d X_i}{\sqrt{\sum_{i=1}^d \E{X_i^2}}}, \,\,\mathcal N(0,1)}=O\left( \beta^{-1}\sum_{i=1}^{d}\frac{1}{n_{i}^{2}} \right)^{1/5}. 
%\end{equation}
\end{prop}
\noindent We prove Proposition \ref{P:alpha-normality} in Section \ref{S:alpha-normality-pf} below. The next result shows that the sum of the conditional expectations in $\sum_i X_i$ contributes a constant $\beta/2$ up to errors of the form $\sum_i n_i^{-2}.$
\begin{prop}\label{P:X_i-pert}
  For any fixed $0<\al<1$, we have
\[\sum_{i=1}^d X_i = \sum_{i=1}^d \ln_\al\lr{\frac{S^{(i)}}{S^{(i-1)}}}\one_{A_{i-1}} + \frac{\beta}{2}+Y,\]
where $Y$ is a random variable satisfying
\[\E{Y}=O\lr{\sum_{i=1}^d n_i^{-2}}.\]
\end{prop}
\noindent Proposition \label{P:X_i-pert} follows from Proposition \ref{p:X_moments} below. To combine Propositions \ref{P:alpha-normality} and \ref{P:X_i-pert}, we will need the following simple result about perturbations under the $KS$-distance. 
\begin{lem}[Properties of $d_{KS}$]\label{L:KS-perturbation}
If $\cN(0,\beta)$ is centered Gaussian with variance $\beta$, $X$ is any random variable, and $Y$ is a positive random variable then there is a universal constant $C$ so that we have:
\begin{equation}
 d_{KS}\left(X+Y,\cN(0,\beta) \right) \leq d_{KS}(X,\cN(0,\beta)) + C \sqrt{\beta^{-1/2}\E{Y}}.\label{E:KS-mean-pert}
\end{equation}
For any $k>0$, there exists a constant $C$ so that
\begin{equation}
 \d_{KS}\left(X(1+k),\cN(0,1) \right) \leq d_{KS}(X,\cN(0,1)) + Ck. \label{E:KS-mult-pert}
\end{equation}
Further, if $X,Y$ are any two random variables on the same probability space, 
then: 
\begin{equation}
 d_{KS}\left(X,Y\right)\leq \p\left(X\neq Y\right). \label{E:KS-TV}
\end{equation}
\end{lem}

\noindent Combining Propositions \ref{P:alpha-normality}, Proposition \ref{P:X_i-pert} and Lemma \ref{L:KS-perturbation}, we obtain
\[d_{KS}\lr{\sum_{i=1}^d \ln_\alpha\lr{\frac{S^{(i)}}{S^{(i-1)}}}\one_{A_{i-1}}, ~\mathcal N(-\beta/2,\beta)}= O\lr{\beta^{-1}\sum_{i=1}^d n_i^{-1}} + O\lr{\beta^{-2}\sum_{i=1}^d n_i^{-2}}^{1/5}+O\lr{\beta^{-1/2}\sum_{i=1}^d n_i^{-2}}^{1/2}.\]
Finally, combining the following estimate with \eqref{E:KS-TV} completes the proof of Theorem \ref{T:main}. 
\begin{lem}\label{L:prob-est}
  For any fixed $0<\al<1$ and any $m\geq 1$ we have
\[\p\lr{\sum_{i=1}^d\ln_\alpha\lr{\frac{S^{(i)}}{S^{(i-1)}}}\one_{A_{i-1}}\neq \ln(S^{(d)})}=O\lr{\sum_{i=1}^d n_i^{-m} + \sum_{i=1}^d \, p^{n_i}}.\]
\end{lem}
\subsection{Proof of Proposition \ref{P:alpha-normality}}\label{S:alpha-normality-pf}
In the proof of Proposition \ref{P:alpha-normality}, we will use the notation
\begin{equation}
 \F^{(i)} \defequal \big|\big|\vec{u}^{(i)}\big|\big|_4^4,\label{E:Fi-def}
\end{equation}
and we will say that a random variable $Y$ is $O_{a.s.}(f(n_{i-1}))$ if there exists $C>0$ independent of $n_i,d$ so that $\abs{Y}\leq Cf(n_{i-1})$ almost surely. The constant $C$ may depend on the moments of the random variable $\mu$ and $p$, which we think of as fixed. To conclude the approximate normality  \eqref{E:alpha-normality} we will use the following theorem. 
\begin{thm}[Special Case of Martingale CLT with Rate \cite{haeusler1988rate}]\label{T:MCLT}
  Suppose that $X_0,X_1,\ld$ is a martingale difference sequence with respect to a filtration $\set{\mathcal F_i,\, i=0,1,\ldots}$. Then
\begin{equation}
d_{KS}\lr{\frac{\sum_{i=1}^d X_i}{\lr{\sum_{i=1}^d \E{X_i^2}}^{1/2}}, ~\cN(0,1)}\leq \left[\frac{\E{\lr{\sum_{i=1}^d \E{X_i^2}-\E{X_i^2~\given{\cF_{i-1}}}}^2}+\sum_{i=1}^d\E{X_i^4}} {\lr{\sum_{i=1}^d \E{X_i^2}}^2}\right]^{1/5}\label{E:Rate-p=2}
\end{equation}
\end{thm}
\noindent The following Proposition allows us to control the $2^{nd}$ and $4^{th}$ moments of $X_i$ appearing on in \eqref{E:Rate-p=2}. 
\begin{prop}\label{P:log-alpha-moments}
For any $0<\al<1$, we have that the conditional 2nd and 4th moments of $X_i$ are:
\begin{align}
\E{{X}_{i+1}^2~\given{\cF_{i}}} &= \lr{\lr{\frac{3}{p}-1}\frac{1}{n_{i+1}} + \frac{\mu_4 - 3}{p n_{i+1}}\frac{F^{(i)}}{\lr{S^{(i)}}^2} + O_{a.s.}(n_{i+1}^{-2})}\one_{A_{i}} \label{E:X-conditional-2nd-moment}\\
\E{{X}_{i+1}^4~\given{\cF_{i}}}&=O_{a.s.}(n_{i+1}^{-2})\one_{A_i}. \label{E:X-conditional-4th-moment}
\end{align}
Moreover, for any $i \geq 1$ and any $j \leq i-1$,
\begin{equation}
 \E{\frac{F^{(i)}}{\left(S^{(i)}\right)^{2}}\one_{A_{i-1}}~\big|~{\cF_j} }\leq\frac{8p^{-1}\mu_{4}-4}{n_{i}} .\label{E:F-S-ratio}
\end{equation}
\end{prop}

\noindent We will prove Proposition \ref{P:log-alpha-moments} in Section \ref{S:log-alpha-moments-pf} below. To complete the proof of Proposition \ref{P:alpha-normality}, note that Proposition \ref{P:log-alpha-moments} yields
\begin{align*}
\var \lr{{X}_{1}}= \E{{X}_{1}^2} &= \lr{\frac{3}{p}-1}\frac{1}{n_{i}} + \frac{\mu_4 - 3}{p n_1} \norm{\vec{u}^{(0)}}_4^4 + O(n_{1}^{-2}),\\
\var\lr{{X}_{i+1}} = \E{{X}_{i+1}^2} &= \lr{\frac{3}{p}-1}\frac{1}{n_{i}} + O(n_{i}^{-2})+ O(n_{i+1}^{-2}), \quad i\geq 1.
 \end{align*}
Hence, in particular, 
\[ \sum_{i=1}^{d} \E{ X_i^2 } = \be + O\lr{\sum_{i=1}^{d} n_i^{-2}}.\]
Thus, \eqref{E:alpha-normality} follows from the previous line together with \eqref{E:X-conditional-4th-moment}, \eqref{E:KS-mult-pert}, and 
\begin{equation}
\E{\lr{\sum_{i=1}^d \E{X_i^2}-\E{X_i^2~\given{\cF_{i-1}}}}^2}=O\lr{\sum_{i=1}^d n_i^{-2}}.\label{E:alpha-normality-num}
\end{equation}
To prove this bound, we begin by using Proposition \ref{P:log-alpha-moments} to establish two inequalities which hold for any fixed $i$:
\begin{align}
 \abs{\E{X_i^2}-\E{X_i^2~\given{\cF_{i-1}}}} \leq& \lr{\frac{3}{p}-1}\frac{1}{n_i}\abs{\one_{A_{i-1}} - \p(A_{i-1})} + \frac{\mu_4 - 3}{p n_i}\abs{ \frac{F^{(i-1)}}{(S^{(i-1)})^2} - \E{\frac{F^{(i-1)}}{(S^{(i-1)})^2}} } + O_{a.s.}(n_i^{-2}) \label{E:diff_bound} \\
  \leq& \frac{\mu_4}{pn_i} + O_{a.s.}(n_i^{-2}) \label{E:diff_bound_n}
\end{align}
Now notice that if $1 \leq j \leq d$ is another index so that $j < i$, then if we take the $\cF_{j-1}$-conditional expectation of equation \eqref{E:diff_bound}, we have by using \eqref{E:F-S-ratio} to bound the expectation of $\frac{F^{(i-1)}}{(S^{(i-1)})^2}$ (along with the elementary fact $\E{|Z-\E{Z}|}\leq 2 \E{Z}$ for positive random variables) and the fact that $\E{\abs{1_{A_i} - \p(A_i)}} = 2\p(A_i)\p(\bar{A_i}) = O(p^{n_i}) = O(n_i^{-1})$ that:
\begin{equation}
\E{ \abs{\E{X_i^2}-\E{X_i^2~\given{\cF_{i-1}}}} \given{\cF_{j-1} } } \leq O\lr{n_i^{-1} n_{i-1}^{-1}} \label{E:cond-ineq}
\end{equation}
With these inequalities in hand, we proceed by expanding the square as follows:
\begin{align*}
 & \E{\lr{\sum_{i=1}^d \E{X_i^2}-\E{X_i^2~\given{\cF_{i-1}}}}^2} 
 =& \E{\lr{\sum_{i,j=1}^d \lr{\E{X_i^2}-\E{X_i^2~\given{\cF_{i-1}}}} \lr{\E{X_j^2}-\E{X_j^2~\given{\cF_{j-1}}}}}}
\end{align*}
The diagonal terms, when $i=j$, are bounded since $\E{\E{X_i^2}-\E{X_i^2~\given{\cF_{i-1}}}}^2 = O(n_i^{-2})$ by the bound in equation \eqref{E:diff_bound_n}. In the remaining off-diagonal terms, by first taking the $\cF_{j-1}$-conditional expectation out via the tower property, we have by the inequality \eqref{E:cond-ineq} that:
 \[ \E{\lr{\E{X_i^2}-\E{X_i^2~\given{\cF_{i-1}}}} \lr{\E{X_j^2}-\E{X_j^2~\given{\cF_{j-1}}}}} = O\lr{ n_i^{-1} n_{i-1}^{-1} n_j^{-1} n_{j-1}^{-1} } \]
 Finally, summing all the bounds for diagonal and off-diagonal entries we see that this entire numerator from equation \eqref{E:Rate-p=2} is bounded by $O(\sum_{i=1}^d n_i^{-2})$, which proves \eqref{E:alpha-normality-num} and completes the proof of Proposition \ref{P:alpha-normality} modulo checking Proposition \ref{P:log-alpha-moments}.

\subsubsection{Proof of Proposition \ref{P:log-alpha-moments}}\label{S:log-alpha-moments-pf} We begin by establishing some preliminary results. 
\begin{lem}\label{L:hatS-moments} Let $n,m \in \bN$ be two layer widths, and let $\vec{u} \in \bR^{m}$ be any non-zero fixed vector. Let $\mat{D} = \text{diag}\lr{\xi_1,\ld,\xi_n} \in \bR^{n \times n}$, $\p\lr{\xi_i = 1} = 1 -\p\lr{\xi_i = 0} = p$ be the diagonal Bernoulli$(p)$ matrix, and $\mat{W} \in \bR^{n \times m}$ be the weight matrix whose entries are iid $W_{i,j} \sim \mu$ for every $1\leq i\leq n, 1\leq j\leq m$. Then:  
\begin{align}
\e\left[1-\norm{\frac{1}{\sqrt{np}} \mat{D}\, \mat{W} \vec{u}}_2^2\bigg/\norm{\vec{u}}_2^2\right] & =0,\qquad \e\left[\lr{1-\norm{\frac{1}{\sqrt{np}}\mat{D}\, \mat{W} \vec{u}}_2^2\bigg/\norm{\vec{u}}_2^2}^2 \right]  =\lr{\frac{3}{p}-1}\frac{1}{n}+\frac{\mu_4 - 3}{p n}\frac{\norm{\vec{u}}_4^4}{\norm{\vec{u}}_2^4} \label{E:hatS-low-moments}
\end{align}
Moreover, with the same setup as above, the following error estimates hold uniformly over all non-zero vectors $\vec{u} \in \bR^m$ (i.e. the constants in the $O$ errors depend only on the moments of $\mu$ and on $p$):
\begin{align}
\e\left[\lr{1-\norm{\frac{1}{\sqrt{np}}\mat{D} \, \mat{W} \vec{u}}_2^2\bigg/\norm{\vec{u}}_2^2}^3\right] &= O(n^{-2}),\qquad
\e\left[\lr{1-\norm{\frac{1}{\sqrt{np}}\mat{D} \, \mat{W} \vec{u}}_2^2\bigg/\norm{\vec{u}}_2^2}^{2r} \right] = O(n^{-2r}), \forall r \geq 2 \label{E:hatS-higher-moments}
\end{align} 
\end{lem}

\begin{proof}
Note that
\[\norm{\frac{1}{\sqrt{np}}\mat{D}\, \mat{W} \vec{u}}_2^2\bigg/\norm{\vec{u}}_2^2\stackrel{d}{=}\frac{1}{n}\sum_{j=1}^nZ_j\]
where $Z_j$ are independent and 
\[Z_j \stackrel{d}{=} p^{-1}\xi_j\lr{\inprod{W_j}{\widehat{u}}}^2,\qquad \widehat{u} \defequal \frac{\vec{u}}{\norm{\vec{u}}},\]
where $W_j$ denotes the $j^{th}$ row of $\mat{W}.$ Since the entries of $W_j$ are iid with mean $0$ and variance $1$, we have $\e[\lr{\inprod{W_j}{\widehat{u}}}^2]= \norm{ \hat{u} }^2 = 1$. Hence $\E{Z_j}=1$ for each $Z_j$ and we conclude 
\[\e\left[\norm{\frac{1}{\sqrt{np}}\mat{D}\, \mat{W} \vec{u}}_2^2\bigg/\norm{\vec{u}}_2^2\right] =1 ,\]
proving the first relation \eqref{E:hatS-low-moments}. Since each $Z_j$ is mean $1$, the relations \eqref{E:hatS-higher-moments} follow by standard esimates of the $3^{rd},{2r}^{th}$ moments of a sum of $n$ iid centered random variables. Finally to check the second relation in \eqref{E:hatS-low-moments}, we write
\[\e\left[\lr{\frac{1}{n}\sum_{j=1}^n Z_j}^2\right]=\frac{1}{n}\sum_{j=1}^n \mathrm{Var}[Z_j]=\mathrm{Var}[Z_1].\]
Moreover, 
\[\e[Z_1^2]=p^{-1}\e\left[\lr{\inprod{W_j}{\widehat{u}}}^4\right]=p^{-1}\sum_{i_1,i_2,i_3,i_4} \prod_{m=1}^4 \widehat{u}_{i_m} \e\left[\prod_{m=1}^4 W_{j,i_m}\right].\]
By direct evaluation, using that $\mathrm{Var}[W_{j,i}]=1$, we find
\[\e\left[\prod_{m=1}^4 W_{j,i_m}\right]=\lr{\mu_4-3}\delta_{\set{i_1=i_2=i_3=i_4}} + \delta_{\left\{\substack{i_1=i_2\\ i_3=i_4}\right\}} + \delta_{\left\{\substack{i_1=i_3\\ i_2=i_4}\right\}} + \delta_{\left\{\substack{i_1=i_4\\ i_2=i_3}\right\}}.\]
Hence, we find
\begin{align*}\mathrm{Var}[Z_1]&=p^{-1}\sum_{i_1,i_2,i_3,i_4} \prod_{m=1}^4 \widehat{u}_{i_m}
\lr{\lr{\mu_4-3}\delta_{\set{i_1=i_2=i_3=i_4}} + \delta_{\left\{\substack{i_1=i_2\\ i_3=i_4}\right\}} + \delta_{\left\{\substack{i_1=i_3\\ i_2=i_4}\right\}} + \delta_{\left\{\substack{i_1=i_4\\ i_2=i_3}\right\}}}-p^{-1}\\  
&=\lr{\frac{3}{p}-1} + \frac{\mu_4-3}{p}\norm{\widehat{u}}_4^4=\lr{\frac{3}{p}-1} + \frac{\mu_4-3}{p}\frac{\norm{\vec{u}}_4^4}{\norm{\vec{u}}_2^2},
\end{align*}
as claimed. 
\end{proof}

\noindent The following corollary immediately yields \eqref{E:F-S-ratio}.
\begin{cor} \label{c:F_over_s2}
For any $0 < \al < 1$ and uniformly over all non-zero vectors $u \in \bR^m$, we have:
\[ \e\left[\norm{\frac{1}{\sqrt{np}}\mat{D}\,\mat{W}\vec{u}}_{4}^{4}\bigg/\norm{\frac{1}{\sqrt{np}}\mat{D}\,\mat{W}\vec{u}}_{2}^{4}\right]\leq\frac{8p^{-1}\mu_{4}-4}{n},\qquad 
\p\lr{ \norm{\frac{1}{\sqrt{np}} \mat{D}\, \mat{W} \vec{u}}_2^2\bigg/\norm{\vec{u}}_2^2 <\al } = \frac{1}{(1-\al)^4} O(n^{-2}).
\]
\end{cor}

\begin{proof}
The tail estimate follows by using the Chebyshev inequality and \eqref{E:hatS-higher-moments}. The bound on the expectation is obtained as follows:
\begin{align*}
\e\left[\frac{\norm{\frac{1}{\sqrt{np}}\mat{D}\,\mat{W}\vec{u}}_{4}^{4}}{\norm{\frac{1}{\sqrt{np}}\mat{D}\,\mat{W}\vec{u}}_{2}^{4}}\right] & =\e\left[\frac{\norm{\frac{1}{\sqrt{np}}\mat{D}\,\mat{W}\vec{u}}_{4}^{4}}{\norm{\frac{1}{\sqrt{np}}\mat{D}\,\mat{W}\vec{u}}_{2}^{4}}\one\left\{ \norm{\frac{1}{\sqrt{np}}\mat{D}\,\mat{W}\vec{u}}_{2}^{2}<\half\norm{\vec{u}_{2}}^{2}\right\} \right]+\e\left[\frac{\norm{\frac{1}{\sqrt{np}}\mat{D}\,\mat{W}\vec{u}}_{4}^{4}}{\norm{\frac{1}{\sqrt{np}}\mat{D}\,\mat{W}\vec{u}}_{2}^{4}}\one\left\{ \norm{\frac{1}{\sqrt{np}}\mat{D}\,\mat{W}\vec{u}}_{2}^{2}\geq\half\norm{\vec{u}_{2}}^{2}\right\} \right]\\
 & \leq\e\left[1\cdot\one\left\{ \norm{\frac{1}{\sqrt{np}}\mat{D}\,\mat{W}\vec{u}}_{2}^{2}<\half\norm{\vec{u}_{2}}^{2}\right\} \right]+\e\left[\frac{\norm{\frac{1}{\sqrt{np}}\mat{D}\,\mat{W}\vec{u}}_{4}^{4}}{\left(\half\norm{\vec{u}_{2}}^{2}\right)^{2}}\cdot1\right]\\
 & \leq\frac{4\left(p^{-1}\mu_{4}-1\right)}{n}+\frac{4p^{-1}\mu_{4}}{n} =\frac{8p^{-1}\mu_{4}-4}{n}.
\end{align*}
\end{proof}

\noindent To complete the proof of Proposition \ref{P:log-alpha-moments}, it remains to check \eqref{E:X-conditional-2nd-moment} and \eqref{E:X-conditional-4th-moment}. To do this, we begin with the following observation. 

\begin{lem}\label{L:log-expand}
Let $\ln_{\al}(x)=\begin{cases}
\ln(\al) & x<\al\\
\ln(x) & x\geq\al
\end{cases}$. Suppose $Y$ is a non-negative random variable. Then there are absolute constants $C$ (here we let $C$ refer to a generic constant which may change value from line to line) so that for any $0<\al<1$:
\begin{align*}
\e\left[\left|\ln_{\al}(Y)-\left\{ (Y-1)-\half(Y-1)^{2}+\frac{1}{3}\left(Y-1\right)^{3}\right\}\right|\right] & \leq\frac{1}{\al^{4}}\e\left[\left(Y-1\right)^{4}\right]+\ln\left(\al\right)\pp{Y<\al}\\
\e\left[\left|\ln_{\al}^{2}(Y)-\left\{ (Y-1)^{2}-\left(Y-1\right)^{3}\right\}\right|\right] & \leq\frac{C+C\ln(\al)}{\al^{4}}\e\left[\left(Y-1\right)^{4}\right]+\ln^{2}\left(\al\right)\pp{ Y<\al}\\
\e\left[\left|\ln_{\al}^{3}(Y)-\left\{ \left(Y-1\right)^{3}\right\} \right|\right] & \leq\frac{C+C\ln(\al)^{3}}{\al^{3}}\e\left[(Y-1)^{4}\right]+\ln^{3}\left(\al\right)\pp{Y<\al}\\
\e\left[\ln_{\al}^{4}(Y)\right] & \leq\frac{C+C\ln^{3}(\al)}{\al^{4}}\e\left[(Y-1)^{4}\right]+\ln^{4}\left(\al\right)\pp{Y<\al}
\end{align*}
\end{lem}
\begin{proof}
The proof is an elementary exercise in the Taylor series expansion for $\ln(x)$ applied to points in the interval $[\al,\infty)$ on which the derivatives of $\ln(x)$ are bounded.
\end{proof}
\noindent Lemma \ref{L:log-expand} together with Lemma \ref{L:hatS-moments} gives the following information on the conditional moments of $X_i,$ which directly allows us to conclude \eqref{E:X-conditional-2nd-moment} and \eqref{E:X-conditional-4th-moment} and hence complete the proof of Proposition \ref{P:log-alpha-moments}.
\begin{prop} \label{p:X_moments}
Recall the vectors $\vec{u}^{(i)}$ and their norms normalized $L^2$ and $L^4$ norms, $S^{(i)}$ and $F^{(i)}$ defined in \eqref{E:S-def} and \eqref{E:Fi-def}. We have for each $i\geq 0$
\begin{align*}
\e\left[\ln_{\alpha}\left(\frac{S^{(i+1)}}{S^{(i)}}\right)\one_{A_{i}}\given{\cF_{i}}\right] & = -\half\e\left[ \lr{\frac{S^{(i+1)}}{S^{(i)}} - 1}^2 \one_{A_i} \given{\cF_i}\right] + O_{a.s.}(n_{i+1}^{-2})\one_{A_{i}}.\\
& = -\half \lr{\lr{\frac{3}{p}-1}\frac{1}{n_{i+1}} + \frac{\mu_4 - 3}{p n_{i+1}}\frac{F^{(i)}}{\lr{S^{(i)}}^2} + O_{a.s.}(n_{i+1}^{-2})}\one_{A_{i}}\\
\e\left[\ln_{\alpha}^{2}\left(\frac{S^{(i+1)}}{S^{(i)}}\right)\one_{A_{i}}\given{\cF_{i}}\right] & =\e\left[ \lr{\frac{S^{(i+1)}}{S^{(i)}} - 1}^2 \one_{A_i} \given{\cF_i}\right] + O_{a.s.}(n_{i+1}^{-2})\one_{A_{i}}.\\
& = \lr{\lr{\frac{3}{p}-1}\frac{1}{n_{i+1}} + \frac{\mu_4 - 3}{p n_{i+1}}\frac{F^{(i)}}{\lr{S^{(i)}}^2} + O_{a.s.}(n_{i+1}^{-2})}\one_{A_{i}}\\
\e\left[\ln_{\alpha}^{3}\left(\frac{S^{(i+1)}}{S^{(i)}}\right)\one_{A_{i}}\given{\cF_{i}}\right] & =O_{a.s.}\left(n_{i+1}^{-2}\right)\one_{A_{i}}\\
\e\left[\ln_{\alpha}^{4}\left(\frac{S^{(i+1)}}{S^{(i)}}\right)\one_{A_{i}}\given{\cF_{i}}\right] & =O_{a.s.}\left(n_{i+1}^{-2}\right)\one_{A_{i}}
\end{align*}
\end{prop}

\begin{proof}
On the event $A_i^c$, both sides of the equation are zero and the equality trivially holds. Therefore we have only to consider what happens on the event $A_i$, where $\vec{u}^{(i)}\neq 0$. By equation \eqref{eq:u_rec}, we have that
\begin{equation}
\frac{S^{(i+1)}}{S^{(i)}} = \frac{\norm{\frac{1}{\sqrt{n_{i+1} p}}D^{(i+1)}W^{(i+1)}\vec{u}^{(i)}}_2^2}{\norm{\vec{u}^{(i)}}_2^2}, \label{E:hatS-def}
\end{equation}
Since we are conditioning on the sigma algebra $\cF_i$, we may think of $\vec{u}^{(i)}$ as a fixed vector and apply Lemma \ref{L:hatS-moments}. To make the equations easier to read, we using the shorthand $\widehat{S} = \frac{S^{(i+1)}}{S^{(i)}}$ and we write $\E{\cdot}$ to mean $\E{\cdot\one_{A_i}\given\cF_i}$. Then we have
\begin{align}
\notag  \abs{ \E{\ln_\alpha\lr{\widehat{S}}} + \half \E{(\widehat{S}-1)^2} } &\leq \abs{\E{\ln_\alpha( \widehat{S} ) - \lr{(\widehat{S}-1) -\frac{1}{2}(\widehat{S}-1)^2 + \frac{1}{3}(\widehat{S}-1)^3}}} \\
\notag&+\abs{\E{(\widehat{S}-1)}} + \abs{\E{\frac{1}{3}(\widehat{S}-1)^3}} \\
& \leq \abs{\E{(\widehat{S}-1)}} + \ln(\alpha) \p(\widehat{S}<\alpha)\\
&\quad + \frac{1}{\alpha^4} \E{\lr{\widehat{S}-1}^4} + \abs{\E{\frac{1}{3}(\widehat{S}-1)^3}}.\label{E:log-1-est}
\end{align}
By Lemma \ref{L:hatS-moments} and Lemma \ref{L:log-expand}, all the terms in equation \eqref{E:log-1-est} are $O_{a.s.}(n_i^{-2})$ which completes the result. A similar argument, combining the moment calculations from Lemma \ref{L:hatS-moments} and the series expansion estimates from Lemma \ref{L:log-expand} in the natural way, gives the higher moments of $\ln_\alpha\lr{\frac{S^{(i+1)}}{S^{(i)}}}$.
\end{proof}

\subsection{Facts about KS distance - Proof of Lemma \ref{L:KS-perturbation}}\label{S:KS-pf}
\begin{proof}[Proof of Lemma \ref{L:KS-perturbation}] Let us use the notation $\cN \dequal \cN(0,1).$ Since 
\[d_{KS}(X,\beta^{1/2}\cN))=d_{KS}(\beta^{-1/2} X,\cN)),\]
we may assume without loss that $\beta=1$ when proving \eqref{E:KS-mean-pert} and \eqref{E:KS-mult-pert}. We begin by checking \eqref{E:KS-mean-pert}. We will show that for every $s > 0$, we have that 
\begin{equation}\label{E:dKS_ineq_s}
\d_{KS}(X+Y,\cN) \leq \d_{KS}(X,\cN) + \sqrt{\frac{2}{\pi}}s + \frac{\E{\abs{Y}}}{s}
\end{equation}
from which \eqref{E:KS-mean-pert} follows by taking $s = \sqrt{ \sqrt{2^{-1} \pi } \E{\abs{Y}}}$ to optimize the inequality.  To begin, by considering the random variables $X,Y,\cN$ all on the same probability space, we have the inequality:
\begin{equation}\label{E:dKS_ineq_2}
 \abs{\p\left(X+Y\leq t\right)-\p(\cN\leq t)}	\leq\abs{\p\left(X+Y\leq t,\abs Y\leq s\right)-\p(\cN\leq t,\abs Y\leq s)}+\p\left(\abs Y\leq s\right)
\end{equation}
We now claim that:
\begin{equation} \label{E:dKS_ineq_3}
\abs{\p\left(X+Y\leq t,\abs Y\leq s\right)-\p(\cN\leq t,\abs Y\leq s)}\leq\d_{KS}\left(X,\cN\right)+\p\left(\abs{\cN-t}\leq s\right) 
\end{equation}
This is proven by examining the two possibilities of the absolute value, $\p\left(X+Y\leq t,\abs Y\leq s\right)-\p(\cN\leq t,\abs Y\leq s)$ and $\p(\cN\leq t,\abs Y\leq s)-\p\left(X+Y\leq t,\abs Y\leq s\right)$ by using the inclusions $\left\{ X+Y\leq t,\abs Y\leq s\right\} \subset\left\{ X-s\leq t,\abs Y\leq s\right\}$  and $\left\{ X+s\leq t,\abs Y\leq s\right\} \subset\left\{ X+Y\leq t,\abs Y\leq s\right\}$  respectively for the two cases. In the first case, consider:
\begin{align*}
\p\left(X+Y\leq t,\abs Y\leq s\right)-\p(\cN\leq t,\abs Y\leq s)=&	\p\left(X-s\leq t,\abs Y\leq s\right)-\p(\cN\leq t,\abs Y\leq s) \\
=& \phantom{+}	\p\left(X\leq t+s,\abs Y\leq s\right)-\p(\cN\leq t+s,\abs Y\leq s)\\
	&+ \p(\cN\leq t+s,\abs Y\leq s)-\p(\cN\leq t,\abs Y\leq s)\\
	\leq & \, \d_{KS}(X,\cN) + \p\lr{\abs{\cN-t} \leq s}
\end{align*}
The inequality for $\p(\cN\leq t,\abs Y\leq s)-\p\left(X+Y\leq t,\abs Y\leq s\right)$ is analogous, and equation \eqref{E:dKS_ineq_3} follows. Equation \eqref{E:dKS_ineq_s} then follows by combining equations \eqref{E:dKS_ineq_2},\eqref{E:dKS_ineq_3}, Markov's inequality $\p(\abs{Y} \leq s) \leq s^{-1} \e{\abs{Y}}$, and the standard fact about Gaussian random variables $\sup_{t\in \R}\p\lr{\abs{\cN-t}\leq s}\leq \sqrt{2^{-1} \pi} s$.

%Hence, for any $s>0$, we find
%\begin{align*}
%  d_{KS}(X,X+Y)  &~=~ \sup_{t\in \R}\abs{\p\lr{X+Y\leq t}  ~-~ \p\lr{X\leq t}} \leq \p\lr{Y>s}~+~\sup_{t\in \R}\abs{\p\lr{X+Y\leq t, \abs{Y}\leq s} - \p\lr{X\leq t}}\\
%&~\leq~ \frac{\e[Y]}{s}~+~ \sup_{t\in \R} \p\lr{\abs{X-t}\leq s} ~\leq~ \frac{\e[Y]}{s}~+~ 2d_{KS}(X,\cN)~+~\sup_{t\in \R} \p\lr{\abs{\cN-t}\leq s}\\
%& ~\leq~ \frac{\e[Y]}{s}~+~2d_{KS}(X,\cN)~+~ Cs 
%\end{align*}
%Taking $s=\sqrt{\e[Y]}$ and using that 
%\[d_{KS}(X+Y,\cN)~\leq~ d_{KS}(X, X+Y)~+~d_{KS}(X,\cN)\]
%completes the proof. 

To verify \eqref{E:KS-mult-pert}, note that there exists a constant $C>0$ so that for every $k>0$ 
\[\sup_{t\geq 0}\p\lr{\frac{t}{1+k}\leq \cN\leq t} ~=~\sup_{t\geq 0}\int_{t/(1+k)}^t e^{-s^2/2}\frac{ds}{\sqrt{2\pi}}~\leq~ k\sup_{t\geq 0}\frac{t}{1+k} \frac{e^{-(t/1+k)^2/2}}{\sqrt{2\pi}}\leq Ck.\]
Hence, 
\begin{align*}
\sup_{t\in \R}\abs{\p\left(\cN\leq t\right)-\p\left(X(1+k)\leq t\right)} & ~\leq~ d_{KS}(X,\cN)~+~\sup_{t\geq 0}\p\left(\frac{t}{1+k}\leq \cN\leq t\right)~\leq~d_{KS}(X,\cN) + Ck.
\end{align*}
Finally to show \eqref{E:KS-TV}, we have
\begin{align*}
\abs{\p\left(X\leq t\right)-\p\left(Y\leq t\right)} & =\abs{\e\left[\one_{X\leq t}-\one_{Y\leq t}\right]} =\abs{\e\left[\left(\one_{X\leq t}-\one_{Y\leq t}\right)\one_{X\neq Y}\right]} \leq \p\lr{X\neq Y}.
\end{align*}
This completes the proof. 
\end{proof}

\subsection{Proof of Lemma \ref{L:prob-est}} We have
\[\p\lr{\sum_{i=1}^d\ln_\alpha\lr{\frac{S^{(i)}}{S^{(i-1)}}}\one_{A_{i-1}}\neq \ln(S^{(d)})}\leq \p\lr{\exists i\,:\,S^{(i)}\leq \al S^{(i-1)},\,\,S^{(i-1)}\neq 0}\leq \sum_{i=1}^d \p\lr{S^{(i)}\leq \al S^{(i-1)},\,\,S^{(i-1)}\neq 0}.\]
As in the proof of Proposition \ref{p:X_moments}, on the event $S^{(i-1)}\neq 0,$ we have that
\[\widehat{S}=\frac{S^{(i)}}{S^{(i-1)}}=\frac{1}{n}\sum_{j=1}^n Z_j,\]
where $Z_j$ are iid random variables each equal in distribution to $p^{-1}\xi_j \inprod{W_j^{(i)}}{\widehat{u}}$ where $\widehat{u}$ is an fixed unit vector. In particular, by Lemma \ref{L:hatS-moments}, $\e{Z_j}=1$ and the higher moments of $Z_j$ are uniformly bounded in terms of the moments of $\mu$ and $1-p.$ In particular, for each $m\geq 1,$ there exists $C_m>0$ depending only on the moments of $\mu$ and $1-p$ so that
\[\e{\abs{\widehat{S}-1}^m}\leq \frac{C_m}{n^{m/2}}.\]
Hence, by the Markov inequality:
\[\p\lr{\widehat{S}<\alpha}\leq \p\lr{\abs{\widehat{S}-1}> 1-\alpha}\leq \frac{C_{2m}}{n_i^m(1-\al)^m}.\]
Hence, by using a union bound and the estimate on the probability $A_i$ in \eqref{E:A-def}, we have for each $m\geq 1,$
\[\p\lr{S^{(i)}\leq \al S^{(i-1)},\,\,S^{(i-1)}\neq 0}~\leq~ \frac{C_{2m}}{(1-\al)^m}\sum_{i=1}^d \frac{1}{n_i^m}+O\lr{\sum_{i=1}^d p^{n_i}}~=~ O\lr{\sum_{i=1}^d \frac{1}{n_i^m}+ p^{n_i}},\]
as desired.

\bibliographystyle{plain}
\bibliography{biblio}

\end{document}